\documentclass[11pt]{article}
\usepackage{amsfonts}
\usepackage{amsmath,amsthm,amscd,amssymb,mathrsfs,setspace}
\usepackage{latexsym,epsf,epsfig}
\usepackage{cancel}
\usepackage{showkeys}
\usepackage{color}
\usepackage[hmargin=1in,vmargin=1in]{geometry}
\usepackage{hyperref}
\usepackage{cite}

\newcommand{\ds}{\displaystyle}

\newcommand{\reals}{\mathbb{R}}
\newcommand{\realstwo}{\mathbb{R}^2}

\newcommand{\Dx}{{\partial_x}}

\newcommand{\cA}{{\mathscr{A}}}

\newcommand{\Dz}{{\partial_z}}

\newcommand{\cD}{\mathscr{D}}
\newcommand{\bA}{\mathbb{A}}

\newcommand{\bP}{\mathbb{P}}
\newcommand{\bX}{\mathbb{X}}
\newcommand{\cF}{\mathscr{F}}

\newcommand{\R}{\mathbb{R}}

\newcommand{\Om}{{\Omega}}

\newcommand{\la}{{\lambda}}

\theoremstyle{plain}
\newtheorem{theorem}{Theorem}[section]
\newtheorem{lemma}[theorem]{Lemma}
\newtheorem{proposition}[theorem]{Proposition}

\newtheorem{corollary}[theorem]{Corollary}

\theoremstyle{remark}
\newtheorem{remark}{Remark}[section]
\numberwithin{equation}{section}
\numberwithin{theorem}{section}
\numberwithin{remark}{section}
\numberwithin{assumption}{section}
\numberwithin{condition}{section}
\usepackage{enumitem}

\begin{document}

\title{Large Deflections of A Flow-Driven Cantilever  \\  with Kutta-Joukowski Flow Conditions}
 \author{\small \begin{tabular}[t]{c@{\extracolsep{1em}} c@{\extracolsep{1em}}c@{\extracolsep{1em}}c}
         Maria Deliyianni & Irena Lasiecka & Justin T. Webster \\
 \it University of Arizona & \it Univ. Memphis & \it UMBC \\
 \it Tucson, AZ & \it Memphis, TN & \it Baltimore, MD\\ 
mdeliy1@math.arizona.edu ~~  &  lasiecka@memphis.edu &~~ websterj@umbc.edu
\end{tabular}}

\maketitle

\begin{abstract} \noindent We consider a canonical flow-structure system modeling airflow over a cantilevered beam.  Flow-beam interactions arise in flight systems as well as alternative energy technologies, such as piezoelectric energy harvesters. A  potential flow, given through a hyperbolic equation, captures the airflow interacting with a beam clamped on one end and free on the other. The dynamic coupling occurs through an impermeability condition across the beam; in the wake the Kutta-Joukowsky flow condition is imposed. Several challenges arise in the analysis, including the unboundedness of the flow domain, lack of interface trace regularity, and flow conditions which give rise to a dynamic and mixed boundary value problem. Additionally, we consider a recent nonlinear model capturing the cantilever large deflections through the effects of inextensibility. We produce a viable underlying semigroup for the model's linearization, which includes a flow regularity theory. Then, exploiting higher regularity nonlinear estimates for the beam, we utilize a semigroup perturbation  to obtain local-in-time strong solutions for the nonlinear dynamics.   
\noindent 
\vskip.15cm

\noindent {\em Key Terms}: mathematical aeroelasticity, cantilever,  flutter, semigroups, mixed boundary conditions, dynamic boundary conditions
\vskip.15cm
\noindent {\em 2010 AMS}: 74F10, 74K10,  35L51, 35M33, 35D30 
\vskip.15cm
\noindent {\em Funding}: The first author was supported by the Croatian Science Foundation under  project  IP-2022-10-2962 while in residence at the University of Zagreb in summer 2025. The second author was partially supported by NSF-DMS 2205508. The third author was partially supported by NSF-DMS 2307538.
\end{abstract}

\section{Introduction} 
{\em Flow-structure interactions}, such as flow-plate and flow-beam systems, have been a topic of scientific interest over the past 80 years. From an engineering point of view this can be traced back to the Cold War era, with classical partial differential equation (PDE) models appearing in the context of projectiles and aircraft components \cite{bolotin,dowellnon}. Other engineering applications include biological systems \cite{huang} and, more recently, alternative energy technologies \cite{energyharvesting,piezosurvey,DOWELL}. Flow-structure interactions typically include a gas flow model (e.g., a potential or inviscid flow) in a higher-dimensional domain, with an elastic body (a plate or beam) occupying a portion of the lower dimensional boundary \cite{recentabhi}. Among this class of models, we find airfoils, projectile paneling, bridges, and axial flow cantilevers---see \cite{dowellnon,bal0,book,survey1} and references therein for more discussion. 

Mathematically, the rigorous theory of PDE solutions and associated properties for flow-structure interactions has been less often analyzed. Older works include integral equation approaches, for instance those described by Tricomi \cite{tricomi}.  These were later revived by Balakrishnan and his coauthors \cite{balshub,shubov,shubov*,ambal}  in an attempt to provide a more comprehensive mathematical theory of aeroelasticity \cite{bal0} (and many references therein). In the present work we take the point of view of Chueshov and his coauthors, who reignited the mathematical interest in coupled flow-structure systems in the 1990s \cite{LBC96,b-c-1}; many of those earlier results have been nicely collected in \cite{springer}. Utilizing semigroup and microlocal analysis techniques, some breakthroughs in well-posedness and regularity for such coupled hyperbolic systems occurred in the 2010s \cite{webster,igorirena,supersonic}. The  more recent works \cite{survey1,book} provide summaries and discussions of recent flow-structure results. 

The interest in flow-structure interactions centers about particular types of instabilities  in aerodynamics. In particular, the {\em flutter phenomenon} is of note, whereby an elastic structure is destabilized by the presence of an ambient airflow of sufficient velocity \cite{book,survey1,dowellnon,bolotin,htw,dowell1}. In this case, the onset of instability is associated with the linear aspects of the model, and the resulting post flutter displacements can only be ``tamed" in the model by the presence of a nonlinear elastic restoring force. In nearly all of the aforementioned mathematical works, the structure in question (beam, plate, or shell) is restricted on all edges---for instance, with clamped or hinged type boundary conditions. However, there are  many physical situations which elicit so called free boundary conditions, where some edges of the beam or plate are unrestricted. In such cases, the flow conditions near the free edge and in the wake are more involved than when the structure is fixed \cite{bal0,balshub,KJ,huang,ambal}. A prominently invoked condition is known as the Kutta-Joukowsky condition (KJC) \cite{K1,K2}. 

We phrase the KJC here \cite{bal0} as taking ``a zero pressure jump off the [structure] and at the trailing edge". This condition has been implemented in aeroelasticity as a mechanism for the  removal of a singularity at a distinguished point in unsteady flow \cite{K1,K2}. From an engineering standpoint the KJC gives results in correspondence with experiment (e.g., in studies of high Reynolds flow models \cite{KJ} and references therein).  In line with the analysis of Balakrishnan and Shubov, described in more detail below, we take this to correspond to a vanishing {\em acceleration potential} in the wake of the plane/axis of the structure. While this  condition can be considered somewhat artificial, providing uniqueness in classical  approaches for Laplace's equation, it has been  rigorously studied and implemented \cite{ambal,K1}.
 In the work at hand, we will consider a {\em clamped-free beam} with the KJC in force off of the structure and an impermeability condition (velocity matching) on the surface of the beam. We consider a physical beam model pertinent to cantilever flutter incorporating nonlinear stiffness effects associated to an {\em inextensibility constraint} \cite{inext1,maria1,maria2}. 

The presence of the free beam boundary condition, in conjunction with the KJC enforced off of the beam, leads to mathematical challenges for recovering integrability at infinity and global regularity for flow solutions. In particular, ``lifting" from the flow boundary to the interior---whether in the time domain or Fourier domain---is particularly challenging. While these issues were  present in the former work of the authors \cite{KJ}, the current work provides a clear and full characterization of flow regularity in 2D, connecting the flow problem with the appropriate mixed elliptic boundary value problem and associated anisotropic spatial regularity. We focus on PDE analysis here and connect to integral equations via the Fourier domain, invoking semigroup theory to exploit the cancellation of low-regularity terms.  With a linear theory in hand, the semigroup allows us to set up a fixed point argument for the nonlinear system. In particular, we invoke a dissipation regularization within the perturbation argument to yield a fixed point for smooth solutions. Then, using the recent regularity analysis of the nonlinear inextensible cantilever \cite{maria2}, the regularization term can be made to vanish, after the construction, with a good global a priori estimate for the inextensible beam. This results in local strong solutions for the fully nonlinear model without any damping or regularization present.

\subsection{Axial Flow-Beam Model}

We consider an inviscid, irrotational, and compressible flow. Let $\phi(x,z,t)$ represent the flow's velocity potential  defined on $\mathbb R^2_+ \equiv \{ z >0\}$. To obtain the flow model, we  linearize Euler's equation in 2D about the state $\mathbf U = \langle U,0\rangle$, where $U$ represents the flow velocity and $0 \le |U| <1$---see \cite{book} for more details and \cite{recentabhi} for a related model. The flow is over-body for an elastic beam residing at equilibrium on $x \in (0,L) \subset \partial \mathbb R^2_+$. We consider $x=0$  (the clamped part of the beam) to be the {\em leading edge} of the cantilever, while $x=L$ is free and will model the cantilever's {\em trailing edge}. 
The boundary condition on the surface of the  beam is given through an {\em impermeability condition}, which effectively encodes an Eulerian-to-Lagrangian variable change in the linearization of the interface \cite{bolotin,dowell1,book}. In this context, we obtain a dynamic Neumann-type condition representing the coupling with the flow. Downstream from the trailing edge and extending in the wake, as well as upstream over the inactive portion of the boundary $x \in (-\infty,0)$, we invoke the KJC \cite{bal0,KJ}. This translates to a dynamic, mixed-type \cite{savare} boundary condition where we specify that the pressure  $[\phi_t+U\phi_x]\big|_{z=0}$ vanishes for $x \notin (0,L)$. We let $w(x,t)$ be the transverse deflection of the flexible part of the boundary $(0,L)$, whose dynamics will be governed by an inertial beam equation; the coupling to the flow occurs through a dynamic pressure, which acts as a distributed force on the surface of beam $(0,L)$. 
Thus we have:
\begin{equation}\label{fullpot}
\begin{cases}
\begin{cases}
(\partial_t + U\partial_x)^2\phi=\Delta_{x,z} \phi & \text{ in } \mathbb{R}_+^2\times (0,T)\\
\partial_z\phi =  (\partial_t+U\partial_x)w & \text { on } (0,L)\\
(\phi_t+U\phi_x) = 0  &\text{ on }  (-\infty,0) \cup (L,\infty)\\
\phi(0)=\phi_0,~~  \phi_t(0)=\phi_1
\end{cases}\\[.8cm]
\begin{cases} \mathbf M(w_{tt})+ \mathbf S(w)=p_0(x)+r_{(0,L)}\gamma_0[(\partial_t+U\partial_x)\phi]& \text{ in } (0,L) \times (0,T) \\
w(x=0)=w_x(x=0) = 0 &  (0,T)\\
w_{xx}(x=L)=0;~~BC_{\mathbf M,\mathbf S}(w(x=L))=0 & (0,T) \\
w(0)=w_0,~~  w_t(0)=w_1 \end{cases}
\end{cases}
\end{equation}
The operator $r_{(0,L)}$ represents the restriction to the elastic domain from $\mathbb R = \{z=0\}$, and $\gamma_0[\cdot]$ is the standard trace operator acting as a restriction on appropriately smooth functions. 
Above, $\mathbf S$ represents the elasticity operator of choice, while $\mathbf M$ provides the associated elastic {\em inertial} operator in the problem. There are several viable options for the latter, noting the references \cite{maria1,htw,inext2,lagleug}; the boundary conditions at the beam's free edge reflect the elasticity and inertial operators. 

\subsection{Large Deflection Cantilever Beams}
We describe both extensible and inextensible beam nonlinearities in this section, since both of these can be accommodated through our theory after the linear system theory is established. In our analysis, {\em we will focus on the more recent quasilinear dynamics arising from inextensibility} \cite{maria1,inext1}. 

We recall that an {\em extensible} beam model is one where the dominant nonlinear effects are those of beam stretching due to its bending. In-axis displacement is prevented at the trailing edge,  inducing stretching with transverse movement---for more details, see \cite{htw}. 
 \begin{equation}\label{Bergerplate} (1-\alpha \partial_x^2)w_{tt}+D\partial_x^4 w +k_0 w_t-b_2||w_x||^2w_{xx} = p(x,t)  \text{ in } (0,L) \times (0,T) 
\end{equation}
Above, $D>0$ is the standard elastic stiffness parameter, $\alpha \ge 0$ is a rotary inertia parameter, and $b_2>0$ is a  parameter measuring the strength of nonlinear effects. The damping parameter is given by $k_0 \ge 0$, while $p(x,t)$ represents the dynamic pressure given from the flow, which in practice will be composed of the static pressure $p_0(x)$ and the dynamic flow-pressure $r_{(0,L)}\gamma_0[\phi_t+U\phi_x],$ as written above in \eqref{fullpot}.
The  model in \eqref{Bergerplate} is a simplification of the model in \cite{lagleug}. 

The recent class of models of interest here are the {\em inextensible} models \cite{maria1,maria2}. These  are particularly suited to cantilevered beams \cite{inext1,inext2}, and hence of primary interest here. In this scenario we enforce that the arc length of the deflecting beam is locally constant; this constraint forces any out-of-place displacement of the tip to result in an in-plane displacement. Then, the dominant nonlinearities manifest as both stiffness and inertial terms. In this analysis, one obtains an  {\em inextensibility condition},  written  as 
$(1+u_x)^2+w_x^2=1.$
Linearizing this relation, and imposing it as a constraint, we obtain a nonlinear dynamics with a good variational structure. In the present work we  neglect nonlinear inertial terms resulting from the above procedure, and we recover standard clamped-free boundary conditions for $w$ in the derivation. Thus we obtain, for the beam PDE in the coupled system,
\begin{equation}\label{dowellnon}
w_{tt}-D\partial_x\big[(w_{xx})^2w_x \big]+D\partial_{xx}\big[w_{xx}\big(1+(w_x)^2\big)\big]=p(x,t)~~ \text{ in } (0,L) \times (0,T) \end{equation}
For the above, we will heavily invoke the results in \cite[Section 4]{maria2} which focuses on a priori estimates and a construction of smooth, i.e., $H^4(0,L) \times H^2(0,L)$ type, solutions for \eqref{dowellnon}.

\subsection{Statement of Main Result and Discussion of Previous Work}

We present two main results in this manuscript, one concerning semigroup generation for the appropriately defined linear model \eqref{flowplate0}, and a second result which leverages the first to produce local-in-time strong solutions for \eqref{fullpot}, with the nonlinearity presented in \eqref{dowellnon}. Specifically, we adopt a regularization inside of a perturbation approach. The final step is passing to zero in the regularization parameter and critically uses the estimates on strong solutions (and differences of solutions) found in \cite{maria2} without any regularization. We refer to \cite{pazy} for the standard definitions of strong and generalized (semigroup) solutions associated to Cauchy problems. 

We now restate the full system of interest here, taking $0 \le |U|<1$:
\begin{equation}\label{fullpot*}
\begin{cases}
(\partial_t + U\partial_x)^2\phi=\Delta_{x,z} \phi & \text{ in } \mathbb{R}_+^2\times (0,T)\\
\partial_z\phi =  (\partial_t+U\partial_x)w & \text { on } (0,L)\\
(\phi_t+U\phi_x) = 0  &\text{ on } (-\infty,0) \cup (L,\infty)\\
\phi(0)=\phi_0,~~  \phi_t(0)=\phi_1
\\
 w_{tt}+D \partial^4_x w + \beta D\big[\partial_{x}^2[w_x^2 w_{xx}]-\partial_x[w_{xx}^2w_x ]\big]=r_{(0,L)}\gamma_0[\phi_t+U\phi_x] & \text{ in } (0,L) \times (0,T)\\
w=w_x = 0 & \text { in }\{0\}\times (0,T)\\
w_{xx}=0;~~w_{xxx}=0 & \text{ in }\{L\} \times (0,T)\\
w(0)=w_0,~~  w_t(0)=w_1 \end{cases}
\end{equation}

We begin with our supporting theorem concerning  an underlying linear semigroup.
\begin{theorem}\label{th:main1} 
Consider the linearized version of the coupled flow-beam system obtained from \eqref{fullpot*} by setting $\beta=0$. Then the associated  Cauchy problem defines a linear operator $\widehat{\mathbb A}$ on an associated energy space $Y$ (defined in Section \ref{Section2}) that generates a $C_0$-semigroup. If  $(\phi_0, \phi_1;w_0,w_1) \equiv y_0$ have $y_0 \in \mathscr D(\widehat{\mathbb A})$, the associated solution $(\phi(t),\phi_t(t);w(t),w_t(t)) \equiv y(t)$ is strong (defined in Section \ref{solutions}), and global in the sense of $y \in C([0,T];\mathscr D(\widehat{\mathbb A})) \cap C^1(0,T;Y)$ for any $T>0$.
\end{theorem}

With the notions of $y$ and $\mathscr D(\widehat{\mathbb A})$ in Section \ref{Section2}, we can then state our main result here for the nonlinear system \eqref{fullpot*}.
\begin{theorem}\label{th:main2}
Consider the nonlinear system \eqref{fullpot*} with $\beta=1$. If the initial data are smooth, i.e., $y_0 \in \mathscr D(\widehat{\mathbb A})$, then there exists a unique local strong solution. The strong solution $y(t)$ satisfies the energy identity given  in \eqref{energyrelation} and remains bounded in the finite energy sense  for any $T>0$, i.e., $y \in C([0,T];Y)$. 
\end{theorem}
\begin{remark}[Locality] \label{locality}
By {\em local} above, we mean for any initial data $y_0 \in \mathscr D(\widehat{\mathbb A})$ there exists a time $T^*=T^*(y_0)$ so that the strong solution has 
$y \in C([0,T^*);\mathscr D(\widehat{\mathbb A})) \cap C^1(0,T^*;Y)$. Alternatively, for any time $T>0$, there  exists a ball $\mathcal B(T) \subset \mathscr D(\widehat{\mathbb A})$ so that if $y_0 \in \mathcal B(T)$ then there is a unique strong solution $y  \in C([0,T];\mathcal B(T))) \cap C^1(0,T;Y)$. 
\end{remark}

To provide some context on Theorem \ref{th:main1}, we point to the work in \cite{balshub,bal0}, which recalled classical approaches to airfoils (and also cantilevers) due to Theodorsen and Tricomi \cite{tricomi}. In particular, those authors connected linear flow-structure analysis to integral equations in $L^p$. However, their work did not address finite-energy spaces associated with a dynamical system, phrased in terms of Sobolev spaces. Their approach relies on the structures of classical flow solvers. Yet, these references treat the appropriate boundary conditions, coming from the engineering literature, for structures with free edges and they elucidate the Kutta-Joukowsky condition. 

Focusing more broadly on flow-plate interactions, in \cite{webster,igorirena}  well-posedness of the standard flow-panel model \cite{dowellnon,bolotin} was resolved in the subsonic case $U<1$. Earlier works of Lasiecka and Chueshov et al. \cite{LBC96,b-c-1} had considered regularized cases for this model and produced both stability and well-posedness results; see \cite{springer,survey1,book} for more discussion. These approaches relied on writing the entire flow-plate dynamics as a Cauchy problem on an appropriately chosen state space $Y$, as we do here, despite the dynamics being non-dissipative due to the nature of the flow-plate coupling. Semigroup methods, however, allow one to exploit formal cancellations at the boundary interface, overcoming a lack of requisite trace regularity for wave dynamics. Finding an equivalent norm for the dynamics, and invoking monotone operator methods, yields  generation of an appropriate semigroup; a posteriori estimates are obtained via the Hardy inequality and control over low frequencies. Sharp estimates on the physically-relevant plate nonlinearities \cite{springer} permits the use of perturbation techniques. The subsonic nature $|U|<1$ of the dynamics is also critical to this argument.

The success of semigroup techniques for subsonic flows \cite{webster}, created an abstract framework for well-posedness in the challenging supersonic case, $|U|>1$ case---an problem open for some 20 years after originally posed by Chueshov. The key insight there was provided through a formal analysis of energies: for the supersonic panel, $|U|>1$,  standard energies as presented in \cite{webster} become unsigned, yet, by a change of state variable $\phi_t \to [\partial_t+U\partial_x]\phi$ (velocity to acceleration potential), one can recover a viable flow energy. The trade-off is that this perspective renders the dynamics non-gradient in the new description, with ill-defined velocity traces at the interface, polluting the energy balance. The dynamics were then decoupled into (i) a dissipative piece and (ii) a perturbation. The dissipative piece  is handled via the theory developed in the subsonic case in \cite{webster}. However, item (ii) is more challenging; the perturbation is a priori undefined, but can be interpreted via duality if the wave velocity trace $r_{(0,L)}\gamma_0[\phi_t]$  resides in some (negative) Sobolev space. The inspiration for pursing this approach comes from hyperbolic theory, where traces behave better than the standard theory predicts \cite{miyatake,miyatake2,sakamoto}.  In \cite{supersonic}, a microlocal argument gives $r_{(0,L)}\gamma_0[\phi_t] \in L^2\left(0,T;H^{-1/2}(\mathbb{R}^2)\right)$, which recovers an energy-level bound on the dynamics. Via the variation of parameters formula, the problem is brought into the context of {\em abstract boundary control} and extrapolated semigroups \cite{redbook}, and a  fixed-point argument gives linear well-posedness. This recent resolution of well-posedness for flow-plate models opened up long-time behavior studies explored in \cite{abhi} and exposited in \cite{LW}. 

In 2010, the second and third authors discussed the KJC with Balarkishnan, and this motivated them to consider the  KJC flow-plate problem. They noted that the ``change of state variable"  described above would be viable with the flow conditions presented in \eqref{fullpot}. This produced \cite{KJ}, which considered a flow-plate problem in arbitrary dimensions, and provided the first analysis of the mixed and dynamic KJC in the context of full wave-PDE solutions.  Balakrishnan's work was driven by analysis of integral equations, whereas in \cite{KJ} the model carried simplified plate boundary conditions and  simplifying assumptions about the domain $\Omega$. In the microlocal domain, the techniques of \cite{supersonic} give rise to a particular (and challenging) inversion of the finite Hilbert transform, the resolution of which was indeed the central focus of \cite{KJ}. 

Finally, we contrast \cite{KJ} with the work at hand. In the present work, we focus on a 1D structural equation, where the cantilevered beam model of interest is delineated and justified \cite{inext1,maria1,inext2}. We require no simplifying assumptions on the geometry of structural domain in order to operatre microlocally. We note that  nonlinear structural effects in \cite{KJ} were given by the classical and well-studied von Karman (extensible) nonlinearity \cite{springer}. However, the recent theory of smooth (local) solutions to the quasilinear fourth order inextensible beam dynamics in \cite{maria2} allow us to analyze \eqref{fullpot} with the inextensible nonlinear stiffness model in \eqref{dowellnon}. Our work here is  salient for applications in piezoelectric energy harvesting \cite{energyharvesting,piezosurvey,DOWELL}, as well as more broadly, since many structures have cross-sections (chords) represented by an elastic cantilever. In studies of elastic flutter, one must have a geometric nonlinearity, i.e., a structural nonlinear restoring force, in order to properly observe limit cycle oscillations \cite{dowell1}. 

{\bf Summary}: The model under consideration is motivated by significant engineering applications in the area of flutter control and alternative energies. On the other hand, the mathematical analysis of the problem provides new insights on challenging questions in the areas of:\begin{itemize} \item (i) mixed PDE problems with Zaremba-type boundary conditions, where the boundary conditions change type along the same boundary edge, leading to compromised elliptic regularity; \item (ii) quasilinear beam equations, requiring modern techniques in nonlinear PDE analysis, where structural and geometric features yield crucial cancellations for a priori estimates; \item and (iii) hyperbolic coupling with beam dynamics, which demands the use of natural estimates for the beam and a full account of boundary regularity for the flow's aeroelastic potential.
\end{itemize}

\subsection{Outline of the Remainder of the Paper}

Following the approach first developed in \cite{KJ}, Section~\ref{Section2}  analyzes the linearized version of the coupled flow-structure system by decomposing the dynamics into a dissipative operator \(\mathbb{A}\) and a perturbation operator \(\mathbb{P}\). The operator \(\mathbb{A}\) is shown to be $m$-dissipative \cite{pazy} on a suitable energy space, hence generating a \(C_0\)-semigroup. Then, building on the techniques in \cite{KJ}, we treat the remaining terms as an unbounded perturbation \(\mathbb{P}\), and recast the system into an abstract boundary control framework. This requires adapting the microlocal analysis and associated estimates from \cite{KJ}. By verifying that the perturbation \(\mathbb{P}\) satisfies the required extrapolated regularity, we can apply the perturbation theory to establish semigroup well-posedness for the full linear system, yielding a generator $\bA + \mathbb P \equiv \widehat \bA: \mathscr D(\bA)\subset Y \to Y$.

In Section~\ref{Section3}, we incorporate the quasilinear stiffness effects that arise from large deflections of the cantilevered beam. To handle the resulting nonlinear dynamics, we introduce a strongly damped version of the beam system and construct a fixed point in the domain of the associated linear operator with strong damping present. The solution is obtained as a fixed point of a nonlinear map defined on a carefully chosen invariant space, where we implement a two-level contraction mapping argument to close  estimates. Using the fixed point result, we then remove the strong damping through a careful limit passage. The latter requires a priori estimates for strong solutions, obtained in the absence of damping, for which we appeal to the estimates  in \cite{maria2}. Finally, we obtain a unique strong solution that is local in the sense of Theorem \ref{th:main2}, but without any damping  in the model. 

\section{Development of A Linear Semigroup} \label{Section2}
We now take $\mathbf M = \mathbf I$ and $\mathbf S=D\partial_x^4$ in \eqref{fullpot}, and focus on the underlying
 linear theory. The abstract setup is taken from \cite{KJ}, adapted from \cite{webster,supersonic}, where we take a convenient translation for the wave dynamics by taking $\mu>0$. It can be removed after the construction of solutions in a standard way---see \cite[Ch. 3, Corollary 3.3.17]{book} and the discussion which follows it. \begin{equation}\label{flowplate0}\begin{cases}
(\partial_t+U\partial_x)^2\phi=\Delta_{x,z} \phi -\mu \phi& \text { in } \realstwo_+ \times (0,T),\\
\partial_z \phi = (\partial_t+U\partial_x)w & \text{ on } (0,L)   \times (0,T),\\
(\partial_t+U\partial_x)  \phi = 0 & \text{ on }\mathbb{R} \backslash [0,L] \times (0,T),\\
\phi(0)=\phi_0;~~\phi_t(0)=\phi_1,\\
w_{tt}+D\partial_x^4w= r_{(0,L)}\gamma_0[\phi_t+U\phi_x] & \text { in } (0,L) \times (0,T),\\
w=w_x=0 & \text{ on } \{x=0\} \times (0,T),\\
w_{xx}=w_{xxx}=0 & \text{ on } \{x=L\} \times (0,T),\\
w(0)=w_0;~~w_t(0)=w_1.
\end{cases}
\end{equation}
One of the central contributions of this manuscript is a clear exposition of the construction of the semigroup associated with the dynamics above.  
\subsection{State Space and Linear Energies}
In line with the critical observations in \cite{supersonic,KJ}, we consider the dynamic states $(\phi, \psi; w, w_t)$, where we make the identification $\psi = \phi_t+U\phi_x$, and $\psi$ is identified with the {\em acceleration potential} of the flow. Utilizing $\psi$ and $w_t$ as test functions for the system, we arrive at the following  linear energies, obtained via Green's Theorem applied to \eqref{flowplate0}:
\begin{align}\label{energies}
E_{pl} = \dfrac{1}{2}\big(||w_t||^2+D||w_{xx}||^2\big),~
E_{fl} =  \dfrac{1}{2}\big(||\psi||^2+||\nabla_{x,z} \phi||^2+\mu||\phi||^2\big),
~\mathcal E(t) =  E_{pl}(t)+E_{fl}(t).
\end{align}
These energies provide the formal energy balance for the system (implementing the KJC)
\begin{equation}\label{energyrelation}
\mathcal E(t)+ U\int_0^t \big( w_x,r_{(0,L)}\gamma_0[\psi]\big)_{L^2(0,L)} dt = \mathcal E(0). \end{equation}
This energy relation demonstrates the dynamics as a sum of a generating piece and a ``perturbation".

Denote $H_*^2 \equiv \{ w \in H^2(0,L)~:~w(x=0)=w_x(x=0)=0\}.$ Then finite energy constraints manifest themselves in the natural requirements on the functions $\phi$ and $w$:
\begin{equation}\label{flowreq}
\phi(x,z,t) \in C([0,T]; H^1(\realstwo_+))\cap C^1(0,T;L^2(\realstwo_+)), \end{equation} \begin{equation}\label{platereq}
w(x,t) \in C([0,T]; H_*^2)\cap C^1(0,T;L^2(0,L)).\end{equation}
The principal state space is taken to be 
\begin{equation} \label{statespace}
Y = Y_{f}\times Y_{b} \equiv \big(H^1(\realstwo_+) \times L^2(\realstwo_+)\big)\times\big(H_*^2 \times L^2(0,L)\big),
\end{equation}
 where the choice of norm on $H^1(\mathbb R^2_+)$ encodes the $\mu>0$ parameter, i.e.:
 \begin{equation}\label{norm}||(\phi,\psi)||^2_{Y_f} \equiv ||\nabla_{x,z} \phi||^2_{0,\mathbb R^2_+}+\mu||\phi||^2_{0,\mathbb R^2_+}+||\psi||^2_{0,\mathbb R^2_+}.\end{equation}
 \begin{remark} When $\mu=0$, we topologize the flow space in the sense of $W_1(\R^2_+)$,  the subspace of $L^1_{loc}(\mathbb R^2_+)$ with the {\em gradient norm}, $||\nabla \phi||_{0,\realstwo_+}$), denoting the homogeneous Sobolev space of order $1$. \end{remark}

\subsection{Formal Definition of Solutions}\label{solutions}
 
In the analysis of the work presented herein we will be making use of semigroup theory which necessitates the definition of various types of solutions. To begin with, a pair of functions $\big(\phi(x,z;t);w(x,t)\big)$ satisfying \eqref{flowreq}--\eqref{platereq} is said to be a {\em strong solution} to \eqref{flowplate0} on $[0,T]$ if 
\begin{itemize}
\item $(\phi_t,w_t) \in L^1(a,b; H^1(\realstwo_+)\times H_*^2)$ for any $(a,b) \subset [0,T]$.
\item $(\phi_{tt},w_{tt}) \in L^1(a,b; L^2(\realstwo_+)\times L^2(0,L))$ for any $(a,b) \subset [0,T]$.
\item $\Delta \phi(t)\in L^2(\realstwo_+)$ and $w(t) \in H^4(0,L)\cap H^2_*$ for almost all $t\in [0,T]$.
\item The boundary conditions $\partial_x^2w(L,t)=0$ and $\partial_x^3w(L,t)=0$ hold for almost all $t \in [0,T]$.
\item The equation ~~$w_{tt}+D\partial_x^4 w =r_{(0,L)}\gamma_0[\phi_t+U\phi_x]$ holds in $L^2(0,L)$ for almost all $t >0$.
\item The equation ~~$(\partial_t + U\partial_x)^2\phi=\Delta_{\mu} \phi$ holds for almost all $t>0$ and almost all $(x,z) \in \realstwo_+$.
\item The equation $r_{(-\infty,0) \cup (L, \infty)}\gamma_0[\phi_t+U\phi_x] = 0$ holds for almost all $t>0$ and almost all $x \in (-\infty,0) \cup (L,\infty)$.
\item The equation $\gamma_0[\partial_z\phi]=\ w_t+Uw_x$ holds for almost all $t>0$ and almost all $x \in (0,L)$. 
\item The initial conditions are satisfied pointwise; that is ~$\phi(0)=\phi_0, ~\phi_t(0)=\phi_1, ~w(0)=w_0, ~w_t(0)=w_1,$ from which we can infer $\psi(0)=\phi_1+U\partial_x\phi_0$.
\end{itemize}

A pair of functions $\big(\phi(x,z;t);~w(x;t)\big)$ is said to be a {\em generalized solution} to  problem (\ref{flowplate0}) on the interval $[0,T]$ if there exists a sequence of strong solutions $(\phi^n(t);w^n(t))$ with some initial data $(\phi^n_0,\phi^n_1; w^n_0; w^n_1)$ such that
$$\lim_{n\to \infty} \max_{t \in [0,T]} \Big\{||\partial_t\phi-\partial_t \phi^n(t)||_{L^2(\realstwo_+)}+||\phi(t)-\phi^n(t)||_{H^1(\realstwo_+)}\Big\}=0$$ and
$$\lim_{n \to \infty} \max_{t \in [0,T]} \Big\{||\partial_t w(t)-\partial_t w^n(t)||_{L^2(0,L)} + ||w(t) - w^n(t)||_{H_*^2}\Big\}=0.$$

From the semigroup point of view, we consider the state $y=(\phi,\psi;w,w_t) \in Y$, with $\psi = \phi_t+U\phi_x$. Generalized solutions will correspond to semigroup solutions for  initial data in $Y$. It is straightforward that  generalized solutions are  {\em weak}, i.e., satisfy variational forms---see \cite{springer,KJ}.

\begin{remark}[Strong Nonlinear Solutions] We note that the above definition of a strong solution is  easily modified to include the nonlinearity associated to inextensibility. In particular, we need only replace the third bullet point above with 
\begin{itemize} \item The equation ~~$w_{tt}+D\partial_x^4 w-D\partial_x\big[(w_{xx})^2w_x \big]+D\partial_{xx}\big[w_{xx}\big(1+(w_x)^2\big)\big] =r_{(0,L)}\gamma_0[\phi_t+U\phi_x]$ holds in $L^2(0,L)$ for almost all $t >0$.\end{itemize} Since beam strong solutions have $w \in H^4\cap H^2_*$, there is no additional modification needed to describe strong solutions for the nonlinear system \eqref{fullpot*} with $\beta=1$, as higher order boundary conditions are recovered in the standard way \cite{maria2}. \end{remark}

\subsection{Semigroup Generation Approach} 
Motivated by the setup in \cite{supersonic}, as well as a preliminary work on the KJC problem \cite{KJ}, we decompose the dynamics into a dissipative piece $\bA$ (with $\sigma=0$ below) and a perturbation piece $\bP$ (with $\sigma=1$ below). We define and use $\Delta_{\mu} \equiv \Delta - \mu I$, and  consider:
\begin{equation}\label{flowplate}\begin{cases}
(\partial_t+U\partial_x)^2\phi=\Delta_{\mu} \phi & \text { in } \realstwo_+ \times (0,T),\\
\partial_z \phi = w_t + \sigma Uw_x & \text{ on } (0,L) \times (0,T)\\
(\partial_t+U\partial_x)  \phi = 0 & \text{ on } [0,L]^c \times (0,T) \\
w_{tt}+D\partial_x^4w= r_{(0,L)} \gamma_0[\phi_t+U\phi_x] & \text { in } (0,L) \times (0,T),\\
w(x=0)=w_x(x=0)=0,~~ w_{xx}(x=L)=w_{xxx}(x=L)=0\\
\end{cases}
\end{equation}
 We will define and characterize $\mathscr D(\mathbb A)$, allowing us to invoke integration by parts on the state space $Y$. This operator $\mathbb A$ will be shown to be dissipative \cite{pazy}. Following dissipativity, the resolvent system will be considered with $\lambda>0$, i.e.
$$(\mathbb A -\lambda I)y = \mathcal  F \in Y,$$
where we will produce a solution $y \in \mathscr D(\mathbb A)$ with associated estimates. This step is exceptioanlly challenging in the context of $\mathbb A$, which will be seen to be degenerate and has the flow component defined on an unbounded domain $\mathbb R^2_+$. Solving the system results in the  {\em maximality} property of $\mathbb A$. With maximality and dissipativity established on the Hilbert space $Y$, the Lumer-Phillips Theorem \cite{pazy} applies, yielding strong and semigroup solutions associated to  \eqref{flowplate} with $\sigma=0$.

We consider the dynamic states $(\phi, \psi; w, w_t)$, we recall the energies:
\begin{align}\label{energies*}
E_{b}(t) & = \dfrac{1}{2}\big(||w_t||^2+D||w_{xx}||^2\big),~~~
E_{f}(t) =  \dfrac{1}{2}\big(||\psi||^2+||\nabla_{x,z} \phi||^2 + \mu ||\phi||^2\big),
\end{align}
and ~$\mathcal E(t)  =  E_{b}(t)+E_{f}(t)$, and the formal energy balance for the system 
\begin{equation}\label{energyrelation*}
\mathcal E(t)+ \sigma U\int_0^t \int_0^L  w_x(x;\tau)\psi (x,0;\tau)dx dt = \mathcal E(0).
\end{equation}
\noindent Note that the boundary term arises from cancellations involving $\phi_t$ in the linear flow structure. 

{\bf Assume now that $\mu>0$ and $|U| < 1$}; we consider \eqref{flowplate} with $\sigma=0$. 
To begin with, let us first define a constituent linear beam operator which will be convenient below: $\cA w = D\partial_x^4 w$   and 
\begin{equation}\label{biharmonicdomain}
\mathscr D(\cA) = \{ w \in H^4(0,L) : w(0) =w_x(0)=0; \; w_{xx}(L)=w_{xxx}(L) = 0\}.
\end{equation} 
From this we make the standard \cite{redbook,htw}  identification $\mathscr D(\cA^{1/2}) = H^2_*,$ topologized by the inner-product $(w,\tilde w)_{H^2_*} = D(w_{xx},\widetilde{w}_{xx})_{L^2(0,L)}$. 
We recall the principal state space: $$Y =  \big(H^1(\realstwo_+) \times L^2(\realstwo_+)\big)\times\big(H_*^2 \times L^2(0,L)\big).$$

From the semigroup point of view, we consider the state $y=(\phi,\psi;w,w_t) \in Y$ with dynamics operator
$\bA: \mathscr D(\bA) \subset Y \to Y$ by
\begin{equation}\label{op-A}
\bA \begin{pmatrix}
\phi\\
\psi\\ w\\ v
\end{pmatrix}=\begin{pmatrix}-U\partial_x \phi+\psi\\- U\partial_x\psi + \Delta_{\mu} \phi \\ v \\ -\cA w+ \gamma_0[\psi] \end{pmatrix},
\end{equation}
with domain $\mathscr D(\bA)$ encoding the boundary conditions, given by
\begin{equation}\label{dom-A}
\mathscr D(\bA) \equiv \left\{ y =  \begin{pmatrix}
\phi\\
\psi\\ w\\ v
\end{pmatrix}\in Y\; \left| \begin{array}{l}
-U \Dx \phi + \psi \in H^1(\R^2_+),~\\ -U \Dx \psi  +\Delta_{\mu} \phi \in L^2(\R^2_+), \\
v \in  H_*^2,~\\
-\cA w + r_{(0,L)}\gamma_0  [\psi] \in L^2(0,L) \\
\psi =0 \text{ on } (-\infty,0)\cup(L,\infty), ~~\Dz \phi = v  \text{ in} ~ (0,L). \end{array} \right. \right\}.
\end{equation}
From here, we  obtain the dynamics of \eqref{flowplate} with $\sigma=0$ in the evolution:
\begin{equation}
 y_t = \mathbb A y ;~~y(0)=y_0=(\phi_0, \psi_0; w_0, w_1) \in Y,
\end{equation}
where $\psi_0 = \phi_1+U\partial_x\phi_0$.

To proceed with the analysis of semigroup generation for $\mathbb A$ on $Y$, we first characterize the domain $\mathscr D(\mathbb A)$ to obtain appropriate elliptic regularity for $\phi$ from the mixed problem.
As we are in the subsonic case, $|U| < 1$, we observe a strongly elliptic problem.  

\begin{theorem}[Domain Characterization]\label{domainchar}
If $(\phi,\psi;w,v) \in \mathscr D(\mathbb A)$, then
\begin{itemize}
\item $\phi_x, ~\psi, ~[\psi-U\phi_x] \in H^1(\mathbb R^2_+)$;
\item $w \in \mathscr D(\mathscr A)$;
\item $\Delta_{\mu} \phi \in L^2(\mathbb R_+^2)$;
\item $r_{(0,L)}\gamma_0[\phi_{xz}] \in H^1(0,L)$ and $\gamma_1[\phi] \in H^{-1/2}(\mathbb R)$ (as an extension of the traditional Neumann trace).
\end{itemize}
\end{theorem}

\begin{proof}
Let us consider the resolvent equation with $\lambda=0$. To that end, let us take $(f_1,f_2;g_1,g_2) \in Y$, in accordance with the description of each component of the generator, and consider:
\begin{align}
-U \phi_x + \psi & = f_1 \in H^1(  \realstwo_+) \label{firstflow*}\\
 -U \psi_x +\Delta_{\mu} \phi &= f_2\in L^2(  \realstwo_+) \label{secondflow*}\\
v &= g_1 \in \mathscr D(\cA^{1/2} ) \\
- \cA w + r_{(0,L)}\gamma_0[\psi] &= g_2 \in L^2(0,L) \label{secondbeam*},
\end{align}
with given boundary conditions (justified a posteriori):
\begin{equation} \label{boundaryconditions*}
\begin{cases}
\Dz \phi = v &\text{in }~(0,L) \\
\psi =0 &\text{in }~ (-\infty,0) \cup (L,\infty).
\end{cases}
\end{equation}
Using the first relation for $f_1$ and the last boundary condition, and restricting, we obtain the identity
$$ U \phi_x = - f_1 \in H^{1/2} ((0,L)^c)$$ via the trace theorem. At this stage we must invoke two elliptic problems to obtain additional information for the regularity of the state functions. To that end, let $\tilde \phi = \phi_x$. We differentiate equation \eqref{firstflow*} twice tangentially (in $x$) and \eqref{secondflow*} once, combine them, and differentiate the first boundary condition in \eqref{boundaryconditions*} tangentially. This yields the elliptic problem in $\Delta_U \equiv \Delta_{\mu}-U^2\partial^2_x$: 
\begin{equation}
\begin{cases}
\Delta_U\tilde \phi = \partial_xf_{2}+U\partial_x^2f_{1} \in H^{-1}(\realstwo_+)\\
\Dz \tilde\phi = \partial_xg_{1} \in H^1(0,L)\\
\tilde \phi=-\frac{1}{U}f_1 \in H^{1/2}([0,L]^c)\end{cases}
\end{equation}

A direct application of Lax--Milgram yields  a variational solution $\tilde \phi = \phi_x \in H^1(\mathbb R^2_+)$ above to the mixed problem. 
 In this instance we have the functions $(\phi, \psi; w, v) \in Y$ in hand already. Thus we identify $\partial_x\phi$ with the unique solution of this variational problem; then we can return to \eqref{firstflow*} to conclude that $\psi \in H^1(\mathbb R^2_+)$, so $\psi_x \in L^2(\mathbb R^2_+)$. Moreover, \eqref{secondflow*} then yields that $\Delta_{U} \phi \in L^2(\mathbb R^2_+)$, and since $\phi_x \in H^1(\mathbb R^2_+)$, $\mu>0$, and $|U|<1$, we have that $\Delta_{\mu} \phi \in L^2(\mathbb R^2_+)$ (and hence $\Delta \phi \in L^2(\mathbb R^2_+)$ as well, since $\phi \in H^1(\mathbb R^2_+)$).

We proceed to read off the further domain criteria:
\begin{itemize}
\item Since $\psi \in H^1(\mathbb R^2_+)$, we have $\gamma_0[\psi] \in H^{1/2}(\mathbb R) \subset L^2(\mathbb R)$. Restricting this to $(0,L)$ we have $r_{(0,L)}\gamma_0[\psi] \in L^2(0,L)$.

\item Using the result above, $r_{(0,L)}\gamma_0[\psi] \in L^2(0,L)$, we can deduce from \eqref{secondbeam*} that $-\cA w \in L^2(0,L)$, and thus with $w \in H^2_*$, it is a standard exercise (recovering the higher order boundary conditions at $x=L$) to obtain $w \in \mathscr D(\mathscr A)$. 
\end{itemize}
Finally, the above regularity permit integration by parts for the variables $\phi$ and $\psi$ (in the appropriate generalized Green's sense) in the dissipativity calculation. In particular, we integrate by parts in $x$ (tangentially) with no boundary contributions at infinity for the variables $\phi, \phi_x \in H^1(\mathbb R^2_+)$, and we utilize Green's identity for $\Delta \phi \in L^2(\mathbb R^2_+)$ since $ \phi \in H^1(\mathbb R^2_+)$. In particular, identifying the Neumann trace $\gamma_1 \in \mathscr L(H^2(\mathbb R^2_+),H^{1/2}(\mathbb R))$ with its extension as a bounded linear operator on $H^{-1/2}(\mathbb R)$, we have the following identity:
\begin{equation}\label{greenss}(\Delta \phi, \widetilde \phi) = -(\nabla \phi, \nabla \widetilde \phi)+\langle \gamma_1[\phi], \gamma_0[\widetilde \phi]\rangle_{H^{-1/2}(\mathbb R) \times H^{1/2}(\mathbb R)},~~\forall ~\widetilde \phi \in H^1(\realstwo_+).
\end{equation}
\end{proof}

\subsection{Dissipativity of $\mathbb A$}\label{dissipativity}
In this section we will show that the operator $\mathbb A : \mathscr D(\mathbb A) \subset Y \to Y$ is dissipative in the sense of \cite{pazy}. 
We recall the following inner product on the state space $Y$: 

For $y = (\phi,\psi;w,v),~\widetilde y=(\widetilde \phi, \widetilde \psi; \widetilde w, \widetilde v) \in Y$, consider
\begin{equation} ( y, \widetilde{y})_Y \equiv  ( \nabla \phi , \nabla \widetilde{\phi} )_{  \realstwo_+} +\mu(\phi,\widetilde\phi)_{\realstwo_+} + (\psi, \widetilde{\psi})_{  \realstwo_+}   + D(\partial_x^2 w, \partial_x^2 \widetilde w )_{(0,L)}  + ( v, \widetilde{v} )_{(0,L)}, 
\end{equation}
where each constituent inner product above is $L^2$ on the domain specified by the subscript.

With this inner product, we have for each $y\in \mathscr D(\mathbb A)$:
\begin{align*}(\mathbb A y, y)_Y =&~  \big(\nabla[-U\partial_x\phi+\psi], \nabla \phi\big) + \mu \big(\psi-U\partial_x\phi,  \phi\big) 
\\& +\big([\Delta-\mu I)\phi-U\partial_x \psi ,\psi\big)  \\
& + D (\partial_x^2 v, \partial_x^2w) +\big(r_{(0,L)}\gamma_0[\psi]-\mathscr Aw,v\big).
\end{align*}
According to the previous section, with $y \in \mathscr D(\mathbb A)$, we have that $\phi, \phi_x, \psi \in H^1(\mathbb R_+^2)$, and $\Delta \phi \in L^2(\realstwo_+)$. Additionally, we have $v \in H^2_*$ and $w \in \mathscr D(\mathscr A)$. This allows us to simplify the terms above in standard inner-products as: 
\begin{align*}(\mathbb A y, y)_Y =&~ -U(\nabla\phi_x,\nabla \phi)+ \big(\nabla \psi, \nabla \phi\big) + \mu \big(\psi,  \phi\big) -\mu U(\partial_x\phi,\phi)
\\& +\big(\Delta \phi, \psi\big) -\mu (\phi,\psi) -U(\psi_x,\psi)+ D (\partial_x^2 v, \partial_x^2w) \\
&+\big(r_{(0,L)}\gamma_0[\psi],v\big)-D\big(\partial_x^4w,v\big)
\end{align*}
Through $H^1(\mathbb R^2_+)$ membership, we observe that total derivative terms vanish
$$(\nabla \phi_x, \nabla \phi) = (\phi_x,\phi)=(\psi_x,\psi) = 0,$$
by invoking integration by parts in $x$ and using the notion of vanishing at $\pm |x|\to \infty$ implied in $H^1$  \cite{springer}. We then invoke generalized Green's Theorem, valid in $\mathscr D(\bA)$. In particular, we have
$$(\Delta \phi,\psi) = -(\nabla \phi, \nabla \psi) - \langle \gamma_1[\phi],\gamma_0[\psi] \rangle_{H^{-1/2}(\mathbb R) \times H^{1/2}(\mathbb R)}.$$
Replacing $\gamma_1[\phi]=-\gamma_0[\partial_z \phi]=-v$ on $(0,L)$ and $\gamma_0[\psi] = 0 $ on $[0,L]^c$, we have
$$(\Delta \phi,\psi) = -(\nabla \phi, \nabla \psi) - (v,\gamma_0[\psi])_{(0,L)}.$$

Thus, combining these calculations, we have obtained
\begin{align}
(\mathbb Ay,y) = 0,~~\forall~ y \in \mathscr D(\mathbb A).
\end{align}
\begin{remark}
We have actually shown that both $\pm \mathbb A$ are dissipative in this scenario. 
\end{remark}

\subsection{Resolvent Equation and Maximality}
The resolvent system for $\lambda > 0$ is given through the following system:
\begin{align}
\psi -U \phi_x- \lambda \phi  & = f_1 \in H^1(  \realstwo_+) \label{firstflow**}\\
 \Delta_{\mu} \phi  -U \psi_x - \lambda\psi &= f_2\in L^2(  \realstwo_+) \label{secondflow**}\\
v-\lambda w &= g_1 \in  H^2_* \\
 - \cA w + \gamma_0[\psi] -\lambda w &= g_2 \in L^2(0,L) \label{secondbeam**},
\end{align}
with boundary conditions:
\begin{align} \label{boundaryconditions**}
\begin{cases}
\partial_z \phi =  v & \text{on }(0,L) \\
~~~ \psi = 0 ~~&\text{on }~(-\infty,0)\cup(L,\infty).
\end{cases}
\end{align}
With $\lambda >0$, we will show that  for the given $\mathcal F = (f_1, f_2;g_1,g_2) \in Y$ we can obtain $y=(\phi, \psi;w,v) \in \mathscr D(\mathbb A)$ (with associated estimates), thereby showing that $\mathbb A$ is maximal; hence, with dissipativity established in Section \ref{dissipativity},  the Lumer-Philips Theorem guarantees $\mathbb A$ generates a $C_0$-semigroup of contractions on $Y$. 

We pause to mention that {\em this is a very delicate analysis}; the elliptic work here constitutes one of the central contributions of this treatment. The strategy involves differentiating the system to obtain a weak $H^{-1}$ system that can be solved variationally for the variable $\lambda \phi-U\phi_x \in H^1(\mathbb R^2_+)$. Then we will reconstruct a full solution $y$ from $\phi_x$ and $\psi$ in $\mathscr D(\mathbb A)$. This will involve the touchy construction of a particular antiderivative of  $\phi_x$ that has prescribed decay properties as $|x|\to \infty$, ensuring  integration by parts is viable in  $x$ for the variables $\phi, \psi,$ and $\phi_x$. 

Consider the state  relation for $\psi$, which motivates the variable $\widehat \phi \equiv D_{\lambda} \phi = \lambda \phi +U\partial_x\phi$, yielding the flow system:
\begin{align}
\psi -\widehat \phi & = f_1 \in H^1(  \realstwo_+)\\[.2cm]
\lambda^2\widehat \phi+2\lambda U\widehat\phi_x-\Delta_{\mu}\widehat \phi+U^2\widehat \phi_{xx} &= -[\lambda^2f_1+2\lambda U \partial_xf_1+U^2\partial_x^2f_1 +D_{\lambda} f_2]\in H^{-1}(  \realstwo_+) \label{phihatRHS}
\end{align}
with boundary conditions:
\begin{align} \label{phihatboundary}
\begin{cases}
\partial_z \widehat \phi = D_{\lambda} v  &\text{ on } (0,L)\\
~~~ \widehat \phi = -f_1 ~& \text{ on } ~ (-\infty,0)\cup(L,\infty).
\end{cases}
\end{align}
We  assume that $v \in H^2_*$ is given, which produces $D_{\lambda}v \in H^1_*(0,L)$, and with $f_1 \in H^1(\mathbb R^2_+)$, we obtain that
$$r_{[0,L]^c}\gamma_0[f_1] \in H^{1/2}((-\infty,0)\cup(L,\infty)).$$
Let us now denote the right hand side of \eqref{phihatRHS} by $F$, namely
\begin{equation}
F\equiv -[\lambda^2f_1+2\lambda U \partial_xf_1+U^2\partial_x^2f_1 +D_{\lambda} f_2] \in \partial_x[L^2(\mathbb R^2_+)] \subset H^{-1}(\mathbb R_+^2).
\end{equation} Then equation \eqref{phihatRHS}, accompanied with the boundary conditions \eqref{phihatboundary}, becomes:
\begin{equation} \label{phi_xstrong}
\begin{cases}
-\Delta_U \widehat \phi +2\lambda U\widehat\phi_x + \lambda^2\widehat \phi = F &\in H^{-1}(\realstwo_+)\\
\widehat \phi=-r_{(L,\infty)} ~\gamma_0[f_1] \equiv G_1 &\in \tilde{H}^{1/2}([0,L]^c)\\
\partial_{z} \widehat \phi = D_{\lambda} v \equiv G_2 &\in \tilde{H}^1_*(0,L). 
\end{cases}
\end{equation}
Here, the space $\tilde{H}^{1/2}([0,L]^c)$ is the Hilbert space of the $H^{1/2}(\mathbb{R} \backslash [0,L])$ functions whose trivial extension on the real line belongs to $H^{1/2}(\reals)$ (a similar definition applies for $\tilde{H}^1_*(0,L)$) \cite{Grisvard}.

We now want to define a weak formulation associated to the problem \eqref{phi_xstrong}. For this problem, we will utilize the setup that appears in \cite{savare} for mixed Zaremba-type  problems. To that end, let us introduce the subspace of $H^{1}(\realstwo_+)$:
\begin{equation*}
V = \{ \zeta \in H^1(\realstwo_+)~: ~r_{(-\infty,0)\cup (L,\infty)}\gamma_0[\zeta] \equiv 0 \}.
\end{equation*}
On this space, we consider the bilinear form $a: V \times V \to \reals$ :
\begin{equation} \label{bilinearform}
a(q,\zeta) = (1-U^2)(q_{x},\zeta_{x})+(q_{z},\zeta_{z}) + \lambda U (q_x, \zeta) +  \lambda U (q, \zeta_x)  + ( \mu +\lambda^2) (q,\zeta).
\end{equation}
In addition, let us define the linear functional $L$ on $V$ described by:
\begin{equation}\label{dualityRHS}
\zeta  \in V \to \langle L, \zeta \rangle_{V' \times V} \equiv \langle F, \zeta \rangle  +(G_2, r_{(0,L)} \gamma_0[\zeta])_{L^2(0,L)},
\end{equation}
where the duality pairing above is justified by recalling the definition of $F$ above and writing
$$\langle F, \zeta \rangle \equiv ([\lambda f_2-\lambda^2f_1,\zeta)_{L^2(\mathbb R^2_+)}-\big(2\lambda U f_1+U^2\partial_xf_1 +Uf_2], \partial_x\zeta\big),$$
of course noting that integration by parts in $x$ is tangential.

With the information introduced above, we are now interested in the variational problem: 

\noindent Find $\tilde \phi \in H^1(\realstwo_+)$ such that:
\begin{equation}\label{phi_xweak}
\begin{cases}
a(\tilde \phi,\zeta )= \langle L, \zeta \rangle_{V'\times V},&\text{for all } \zeta \in V \\
\tilde \phi - \tilde{G}_1 &\in V,
\end{cases}
\end{equation}
where $\tilde{G}_1 \in H^1(\realstwo_+)$ is an interior function with $G_1 \in \tilde{H}^{1/2}(\mathbb{R} \backslash [0,L])$  such that:
\begin{equation}\label{Gtilde}
r_{(-\infty,0)\cup (L,\infty)} \gamma_0 [\tilde{G}_1] = G_1,
\end{equation}
as guaranteed by the trace theorem. 

\begin{corollary} \label{Corollaryphi}
There exists $\phi \in V$ that satisfies $a(\phi,\zeta)= \mathcal{L}(\zeta)~\text{for all } \zeta \in V$, where $a$ is the bilinear form introduced in \eqref{bilinearform}  and $\mathcal{L}(\zeta)=\langle L, \zeta \rangle - a(\tilde{G}_1,\zeta).$
\end{corollary}
\begin{proof}
It is direct to show that the bilinear form is continuous and coercive on $V$. In addition, the right hand side is a continuous linear functional on $V$:
\iffalse
\begin{enumerate}
\item {\bf $\bf{a}$ is continuous:}
\begin{align*}
|a(q,u)| & = |(1-U^2)(q_{x},u_{x})+(q_{z},u_{z}) - \mu (q,u)|\\[1em] & \leq (1-U^2)|q_{x}|~|u_{x}|+|q_{z}|~|u_{z}|+\mu|q|~|u| \\[1em]
& \leq C||q||_{1}~||u||_{1}.
\end{align*}

\item {\bf $\bf{a}$ is coercive:}
\begin{align*}
a(q,q) = (1-U^2)(q_{x},q_{x})+(q_{z},q_{z}) + \mu (q,q)| \geq C||q||_{1}.
\end{align*}

\fi

\begin{align*}
|\mathcal{L}(\zeta)| & =\big| \langle F, \zeta \rangle_{V'\times V} +\langle G_2, r_{(0,L)} \gamma_0[\zeta] \rangle_{L^2(0,L)} - a(\tilde{G}_1,\zeta)\big| \\
& \leq ||F||_{V'}~||\zeta ||_{V} + ||{G}_1||_{\tilde H^{1/2}(\mathbb{R} \backslash [0,L])}~||\nabla \zeta ||_{L^2(\mathbb R^2_+)}+||G_2||_{L^2(\mathbb R^2_+)}~||\zeta||_{L^2(\mathbb R^2_+)}\\
& \leq C||\zeta||_{V}.
\end{align*}

With these statements established, Lax--Milgram guarantees that there exists a unique $\phi \in V$ that satisfies the variational relation.
\end{proof}
We immediately have the desired corollary.
\begin{corollary}
The function $\tilde{\phi} \equiv \phi + \tilde{G}_1$ satisfies the problem \eqref{phi_xweak}, where $\phi$ is the function we obtained in Corollary \ref{Corollaryphi} and $\tilde{G}_1$ is described by \eqref{Gtilde}.
\end{corollary}

This above result yields a (weak) solution for $\widehat \phi$ in \eqref{phi_xstrong}:
\begin{equation} \label{phihat}
\widehat \phi = \lambda \phi +U\partial_x\phi \in H^1(\mathbb R^2_+).
\end{equation}
From here, we can immediately define $\psi \equiv f_1 + \hat \phi \in H^1(\mathbb R^2_+)$ from \eqref{firstflow**}. With this information in hand,  solving and acquiring the desired regularity for $w$ and $v$ is direct. Indeed, with $\psi \in H^1(\mathbb R^2_+)$ (and so $r_{(0,L)}\gamma_0[\psi] \in H^{1/2}(0,L)$), we can directly solve \eqref{secondbeam**} for $w$; with $w$ in hand, we algebraically obtain $v$. It remains to recover the variable $\phi \in H^1(\mathbb R^2_+)$. 

Having obtained $\psi \in H^1(\realstwo_+)$, we have $\phi_x \in L^2(\mathbb R^2_+)$. We then specify that $\Delta \phi = f_2+\psi_x \in L^2(\mathbb R^2_+)$. We must use this information to obtain a viable expression for $\phi$. We will produce a particular antiderivative for $\phi$ coming from the differential equation
$$D_{\lambda}\phi =\psi - f_1 \in H^1(\realstwo_+).$$ Indeed we will choose an antiderivative that ensures that $\phi$ decays at $|x|\to \infty$, permitting integration by parts tangentially. With the ability to perform integration by parts, we will be able to show $\phi$, and subsequently $\phi_z$, are in $L^2(\mathbb R^2_+)$; finally, once $\phi \in L^2(\mathbb R^2_+)$, we will immediately recover that $\phi_x \in L^2(\mathbb R^2_+)$.

\begin{lemma} \label{Lemma2.4}
Let $\widehat \phi \in H^1(\mathbb R^2_+)$ be a weak solution as in \eqref{phi_xstrong}. Then there exists $\phi \in H^1(\mathbb R^2_+)$ satisfying $\widehat \phi = \lambda \phi +U\partial_x\phi \in H^1(\mathbb R^2_+).$
\end{lemma}

\begin{proof} Take $\widehat \phi \in H^1(\mathbb R^2_+)$ as the weak solution described in \eqref{phi_xstrong}. We set
$$e^{\lambda x}\widehat \phi(x,z)=\dfrac{d}{dx}\left[e^{\lambda x}\phi(x,z)\right],$$
and solve for an unknown $\phi$ a.e. $x$. 
Consider the particular anti-derivative given by:
\begin{align} \label{phiexpression} \phi(x,z)=
 \frac{1}{U} \int_{-\infty}^{x}e^{-\frac{\lambda}{U}(x-\xi)}\widehat \phi(\xi, z) d\xi.
\end{align}

One can directly verify that the given $\phi$ satisfies the differential equation \eqref{phihat}.  Thus, the only thing left to prove for the corollary is that $\phi$ as defined above carries the desired regularity. We proceed in steps for clarity. 
\vskip.1cm
\noindent {\bf Step 1}: We show that, for a.e. $z$ fixed, $\phi$ tends to zero as $|x|\to \infty$.

\vskip.1cm \noindent The expression $\phi(0,z)$ is well defined a.e. $z$ by the convergent integral 
$$ \frac{1}{U} \int_{-\infty}^{0}e^{\frac{\lambda}{U}\xi}\widehat \phi(\xi, z) d\xi,$$ since $\widehat \phi \in H^1(\mathbb R^2_+)$, via Fubini's Theorem. We will utilize the notation $\phi(0,z)$ and consider
\begin{align*}
\lim_{x \to \infty} \phi(x,z)= \lim_{x \to \infty}  \left( e^{-\frac{\lambda}{U} x}\phi(0,z) + \frac{1}{U}\int_{0}^{x}e^{-\frac{\lambda}{U}(x-\xi)}\widehat \phi(\xi, z) d\xi \right).
\end{align*}
We break the integral term and write
\begin{equation} \label{expressionforz}
\phi(x,z)=e^{-x}\phi(0,z) +  \frac{e^{-\frac{\lambda}{U} x}}{U}\int_{0}^{\frac{x}{2}}e^{\frac{\lambda}{U}\xi}\widehat \phi(\xi, z) d\xi+ \frac{e^{-\frac{\lambda}{U} x}}{U}\int_{\frac{x}{2}}^{x}e^{\frac{\lambda}{U}\xi}\widehat \phi(\xi, z) d\xi.
\end{equation}
Squaring \eqref{expressionforz} and using Jensen's inequality yields
$$\phi^2(x,z)\leq e^{-2x}\phi^2(0,z,t)+ \frac{e^{-\frac{2 \lambda}{U} x}}{U}\int_{0}^{\frac{x}{2}}e^{\frac{2 \lambda}{U}\xi}\widehat \phi^2(\xi, z) d\xi + \frac{e^{-\frac{2 \lambda}{U} x}}{U}\int_{\frac{x}{2}}^{x}e^{\frac{2 \lambda}{U}\xi}\widehat \phi^2(\xi, z) d\xi. $$
Since $\phi(0,z)$ is finite for a.e. $z$, the first term above clearly converges to zero as $x \to \infty$. For the remaining two terms bound them as follows:
\begin{align*}
 \frac{e^{-\frac{2 \lambda}{U} x}}{U}\int_{0}^{\frac{x}{2}}e^{\frac{2 \lambda}{U}\xi}\widehat \phi^2(\xi, z) d\xi +\frac{e^{-\frac{2 \lambda}{U} x}}{U}\int_{\frac{x}{2}}^{x}e^{\frac{2 \lambda}{U}\xi}\widehat \phi^2(\xi, z) d\xi
\leq \frac{e^{- \frac{\lambda x}{U}}}{U}||\widehat \phi(\cdot, z)||^2_{H^1(\mathbb R_x)} + \int_{\frac{x}{2}}^{x}\widehat \phi^2(\xi, z) d\xi.
\end{align*}
The first term on the RHS above tends to zero as $x \to \infty$, while the second term corresponds to a tail of an $H^1(\mathbb R)$ function (of course a.e. $z$). Hence we have obtained 
$$\lim_{x \to \infty} \phi(x,z)=0 \hskip1cm a.e.~z.$$

We now consider $x <0$. 
\begin{equation*}
\left(\int_{-\infty}^{x}e^{-\frac{\lambda}{U}(x-\xi)}\widehat \phi(\xi, z) d\xi\right)^2 \leq  \int_{-\infty}^{x} e^{-\frac{2\lambda}{U}(x-\xi)}\widehat \phi^2(\xi, z) d\xi = e^{-\frac{2\lambda}{U}x}\int_{-\infty}^{x} e^{\frac{2U}{\lambda}\xi}\widehat \phi^2(\xi, z) d\xi \le \int_{-\infty}^x \widehat \phi^2(\xi,z)d\xi.
\end{equation*}
As before, the latter represents the tail of a convergent integral, and hence as $x \to -\infty$ we observe that $\phi(x,z) \to 0$. 

\vskip.3cm
\noindent {\bf Step 2}: We show that $\phi \in L^2(\mathbb R^2_+)$.
\vskip.1cm \noindent
Having shown that  $\phi$ has pointwise values that vanish at infinity, we invoke integration by parts (tangentially) in $x$ to note that 
$$U(\partial_x \phi,\phi)_{L^2(\mathbb R^2_+)}=0.$$ Then, by testing $\widehat \phi = \lambda \phi +U\phi_x$ by $\phi$ and integrating over $\mathbb R^2_+$ we have:
\begin{align*}
\lambda ||\phi||_{L^2(\mathbb R^2_+)}^2 = (\widehat \phi, \phi)_{L^2(\mathbb R^2_+)} \le ||\hat \phi||_{L^2(\mathbb R^2_+)}||\phi||_{L^2(\mathbb R^2_+)}, 
\end{align*}
and since $\lambda>0$ we have that $\phi \in L^2(\mathbb R^2_+)$.

\vskip.3cm
\noindent {\bf Step 3}: We observe that $\phi_x \in L^2(\mathbb R^2_+)$.
\vskip.1cm \noindent
Combining the regularity obtained in Step 2 and the fact that $U\phi_x=\widehat \phi-\lambda \phi$, we  deduce immediately that $\phi_x \in L^2(\mathbb R^2_+)$.
\vskip.3cm
\noindent
{\bf Step 4}: We observe that $\phi_z \in L^2(\mathbb R^2_+)$.
\vskip.1cm
\noindent We can differentiate the relation \eqref{phihat} in $z$. This generates the ODE
\begin{equation}\label{ODEforz}
\lambda \phi_z +U\partial_x\phi_z =\hat \phi_z \in L^2(\mathbb R^2_+).
\end{equation}
We can then verify that the solution $\phi_z$ of \eqref{ODEforz} vanishes for $|x|\to \infty$, since the calculations we demonstrated in Step 1 are applicable with any given function in  $L^2(\mathbb R^2_+)$ rather than $H^1(\mathbb R^2_+)$---as we only require the integrability of tails. Hence, repeating the arguments in Step 3, we observe that $\phi_z \in L^2(\mathbb R^2_+)$.

Thus we have shown that with $v \in H^2_*$ given, we obtain $\widehat \phi$ from  \eqref{phihatboundary}, from which we obtain $\phi \in H^1(\mathbb R^2_+)$ as given by \eqref{phiexpression}. Lastly, we define $\psi \equiv f_1+\widehat \phi = f_1 + \lambda \phi + U\phi_x \in H^1(\mathbb R^2_+).$ Whence we can read off 
$$\Delta_{\mu}\phi=f_2+D_{\lambda}\psi \in L^2(\mathbb R^2_+),$$
Then, with $\phi \in H^1(\mathbb R^2_+)$, Greens Theorem is valid for $\phi$, as in \eqref{greenss}. It is then an exercise to read off the boundary conditions for $\phi$ considering both \eqref{phiexpression} and the boundary conditions for \eqref{phihatboundary}. 

Now, given $\psi \in H^1(\mathbb R^2_+)$ one can immediately solve the biharmonic problem
\begin{align}
v-\lambda w &= g_1 \in H^2_* \\
 - \cA w + \gamma_0[\psi] -\lambda w &= g_2 \in L^2(0,L) \label{secondbeam***}
\end{align}
in a standard fashion, producing $w \in \mathscr D(\mathscr A)$ and $v \in H^2_*$.

A standard  linear fixed point argument (iterating $v \mapsto \psi \mapsto v$) will obtain from the above solvers, and the result is an element $y = (\phi, \psi;w,v) \in \mathscr D(\mathbb A)$ which satisfies the resolvent system
\eqref{firstflow**}--\eqref{secondbeam**}, as desired. 
\end{proof}

Given the dissipativity of $\mathbb A$ as given in Section \ref{dissipativity}, and its maximality demonstrated in this section, we have shown that $\mathbb A: \mathscr D(\mathbb A) \subset Y \to Y$ is $m$-dissipative. We conclude the section by stating this as a theorem.
\begin{theorem} \label{GenerationofA}
Let $\bA$ be the operator given by \eqref{op-A}--\eqref{dom-A}. Then, $\bA$ is  $m$-dissipative on $Y$.  Hence, via the Lumer-Philips theorem, it generates a C$_0$-semigroup of contractions.
\end{theorem}

\subsection{Perturbation $\mathbb P$}
We now encode the flow boundary conditions abstractly into our operator representation of the evolution. We follow the approach developed in \cite{KJ}, namely, we return to the full dynamics $${y}_t = (\bA +\bP)y, ~y(0)=y_0 \in Y$$ representing \eqref{flowplate} with $\sigma=1$, and with $\bA$ as defined in \eqref{op-A}--\eqref{dom-A}.

We proceed through several definitions and lemmas which are found in \cite{KJ}. We  work with anisotropic Sobolev spaces in characterizing the dynamic Neumann ``lift" that captures $\mathbb P$:
\begin{equation} \label{anisotropic}
H^{r,s}(D) = \left \{ f \in H^s(D),  \partial_x^rf \in H^s(D) \right \}.
\end{equation}

Now we will define $\bA_0: \mathscr D(\bA_0) \subset Y_{f} \to Y_{f}$, where $Y_{f}$ is given by \eqref{statespace} with norm \eqref{norm}, as
$$\bA_0 \begin{pmatrix}\phi\\\psi\end{pmatrix} \equiv   \begin{pmatrix}- U \phi_x + \psi\\ - U \psi_x + \Delta_{\mu} \phi \end{pmatrix}.$$ 
The operator $\bA_0$ has domain given by  
\begin{equation}\label{dom-A0}
\mathscr D(\bA_0) \equiv \left\{\begin{pmatrix}
\phi\\
\psi
\end{pmatrix}\in Y_{f}\; \left| \begin{array}{l}
- U \phi_x + \psi \in H^1(\realstwo_+),~\\ U \psi_x + \Delta_{\mu} \phi\in L^2 (\realstwo_+), \\
\partial_z \phi =0 \text{ on }~ (0,L) \\
\psi =0 \text{ on } (-\infty,0)\cup (L,\infty). \end{array} \right. \right\},
\end{equation}
where $\Delta_{\mu} \equiv \Delta - \mu I$. 
\begin{remark} We note that the Neumann trace has meaning in the context of $\mathscr D (\mathbb A_0)$, as in the case of Theorem \ref{domainchar}, and hence the definition above is self-consistent. \end{remark}
In this case, the operator $\mathbb A_0$ is skew-adjoint, as shown in \cite{KJ}, with a complementary characterization of the Dirichlet trace presented below.
\begin{lemma} \label{A0skewadj}
The operator $\bA_0$ defined on $Y_f$ above is skew-adjoint. Moreover,  for $(\phi, \psi) \in \cD(\bA_0^*)$,  we have the identification 
 $N^*[ \bA_0^* - I ] \begin{pmatrix} \phi \\  \psi\end{pmatrix} = \gamma_0 [\psi]$ on $(0,L)$.  
\end{lemma}

With the mapping $\bA_0$ in hand,  we proceed by defining the flow-Neumann map $N:L^2(\mathbb R) \rightarrow Y_{f}$  as follows:
\begin{equation} \label{DefinitionofN}\begin{pmatrix} 
\phi\\
\psi
\end{pmatrix} = Ng , ~~~\text{ if and only if }~~ -U \phi_x + \psi - \phi  =0~~ \text{ and }~ - U \psi_x + \Delta_{\mu} \phi - \psi  =0~ \text{ in }~   \realstwo_+,$$$$ \text{ with }
 \partial_z \phi = g \text{ on } (0,L)  ~~\text{and } \psi =0 \text{ in } [0,L]^c .
\end{equation}
One of the main insights in \cite{KJ} was the utility of the operator $\bA_0 $ over the the typical harmonic extensions associated to   elliptic operators. From that reference, we note the following regularity for \eqref{DefinitionofN}, which  follows from the strong ellipticity of the operator $\Delta_U=\Delta_{\mu}-U^2\partial^2_x$ in the subsonic case  $0\le |U| <1$.
\begin{lemma} The operator $N$, defined above, satisfies
$$N\in \mathscr{L} \big(H^{1,-1/2}(0,L) \rightarrow H^{1,1}(  \realstwo_+) \times H^{0,1}(  \realstwo_+)\big).$$
\end{lemma}

With the definitions of $\bA_0$ and $N$, we are now ready to to  represent the dynamics of \eqref{flowplate} with $\sigma=1$ in an operator form  in that encodes the boundary conditions.

 \begin{equation}\label{op-AS}
\bA \begin{pmatrix}
\phi\\
\psi\\ u\\ v
\end{pmatrix}=\begin{pmatrix} \bA_0  \Big[\begin{pmatrix}\phi\\ \psi\end{pmatrix}  - N v \Big]  + Nv 
 \\ v \\ -\cA u + N^* (\bA_0^*  - I )  \begin{pmatrix}\phi\\\psi \end{pmatrix} \end{pmatrix},~~~~
 \bP \begin{pmatrix}
\phi\\
\psi\\ u\\ v
\end{pmatrix}=\begin{pmatrix} -U (\bA_0-I)N u_x 
 \\ 0 \\ 0  \end{pmatrix}.
\end{equation}
Above, we assign meaning to the operator $(\bA_0-I)N$ via duality (in the sense of $[\mathscr D(\mathbb A_0^*)]'$). We can now
readily verify that $\bA+\bP$ formally encodes the dynamics of \eqref{flowplate}; note that we can write
\begin{align*}
(\bA+\bP)y=&\begin{pmatrix} \bA_0\Big[\begin{pmatrix} \phi \\ \psi\end{pmatrix}-N(v+Uu_x)\Big]+N(v+Uu_x) \\ v \\ -\cA+N^*(\bA^*_0-I)\begin{pmatrix}\phi\\\psi \end{pmatrix} \end{pmatrix}.
\end{align*}

\subsection{Generation for the Full Dynamics: $\bA+\bP$}\label{fulldyn}
We focus on  the abstract Cauchy formulation of \eqref{flowplate}, introduced in the previous section.
To obtain semigroup well-posedness of this problem, we will utilize a linear fixed-point argument through the variation of parameters formula. Relabel the action of the operator $\bP$ as
\begin{equation}\label{op-P}
\bP\begin{pmatrix}\phi\\\psi\\w \\v \end{pmatrix}= \bP_{\#}(w)\equiv\begin{pmatrix} -U (\bA_0-I)N w_x \\0\\0\end{pmatrix},
\end{equation}
and introduce the following inhomogeneous problem:
\begin{equation} \label{inhomcauchy}
\dot{y} = \bA y +\bP_\# \tilde {w},~y(0)=y_0,
\end{equation}
for a given $\tilde{w} \in H_*^2$. The action of $\mathbb P$ can be made precise through duality as follows: 

Let $y,z \in Y$ with $z= (\tilde \phi,\tilde \psi;  \tilde w, \tilde v)$. Then
\begin{align}\label{p-y-z}
(\bP y, z)_Y =& (\bP_{\#}(w),z)_Y
= -U\left((\bA_0-I)Nw_x, \begin{pmatrix} \tilde \phi\\\tilde \psi \end{pmatrix}  \right)
=-U\langle \partial_x w, \gamma[\tilde \psi]\rangle_{(0,L)}.
\end{align}
Since $w_x \in H_*^1(0,L)$, we will require that $\gamma[\tilde \psi] \in H_*^{-1+\epsilon}(0,L)$. Thus, the following lemma from \cite{KJ}, proved via microlocal analysis, is crucial for the continuation of the abstract analysis of the dynamics.

 \begin{lemma}[{\bf Flow Trace Regularity}]\label{le:FTR0} 
Consider $\phi$ to be a given weak solution \cite{springer} to the following hyperbolic problem in $\realstwo_+$ with $f \in C([0,T];L^2(0,L))$: \begin{equation}\label{flow1}\begin{cases}
(\partial_t+U\partial_x)^2\phi=\Delta_{\mu} \phi, & \text { in }~ \realstwo_+ \times (0,T),\\
\phi(0)=\phi_0;~~\phi_t(0)=\phi_1,\\
\Dz \phi = f,& \text{ on }~ (0,L) \times (0,T)\\
\gamma_0[\phi_t+U\phi_x]=\gamma_0[\psi]=0, &~~~~~x \in (-\infty,0)\cup (L,\infty),
\end{cases}
\end{equation}
with $0 \le U <1$. Then for any $\epsilon>0$: $$\partial_t\gamma_0[\phi],~~ \partial_x\gamma_0[\phi]  \in L^2(0,T;H^{-1/2-\epsilon}(\reals))~~~~\forall\, T>0.
$$
Moreover,  with $\psi = \phi_t + U \phi_x $ we have
\begin{equation}\label{trace-reg-est-M0}
\int_0^T\|r_{(0,L)}\gamma[\psi](t)\|^2_{H_*^{-1/2 -\epsilon} (0,L)}dt\le C\left(
E_{fl}(0)+
 \int_0^T\| f(t) \|_{L^2(0,L)}^2dt\right).
\end{equation}
\end{lemma}
\noindent We note that this ``hidden regularity" result is the another of the central contributions of \cite{KJ}, utilizing the invertibility properties of the finite Hilbert transform, as critically discussed in \cite{bal0}. 

We now consider a semigroup mild solution \cite{pazy}, given by the variation of parameters formula for a given $\overline w$:
\begin{equation}\label{solution1}
y(t)=e^{\bA t}y_0 + \int_0^t e^{\bA(t-s)}\bP_{\#} (\overline{w}(s)) ds,
\end{equation}
and re-expressed \cite{redbook} for $\la\in\R\setminus \{0\}$ as:
\begin{equation}\label{solution2}
y(t)=e^{\bA t}y_0 + (\lambda  -\bA)\int_0^t e^{\bA(t-s)}(\lambda  - \bA)^{-1}\bP_{\#} (\overline{w}(s))  ds,
\end{equation}
taken in the sense of $[\cD(\bA^*)]'=[\cD(\bA)]'$ (via the skew-adjointness of $\bA$); this expression is a semigroup solution in the extrapolated sense, and hence $\mathbb P$ is an unbounded perturbation (as in abstract boundary control \cite{redbook}).
\begin{theorem}
Let $T>0$ be fixed, $y_0 \in Y$ and $\overline w \in C([0,T]; H_*^{2}(0,L))$.
Then  the function \begin{equation}\label{mild-sol}
 \ds y(t)=e^{\bA t}y_0+L(\overline w)(t)\equiv
 e^{\bA t}y_0+\int_0^te^{\bA (t-s)}\bP_{\#}(\overline w(s)) ds
 \end{equation}
 belongs to the class $C([0,T]; Y)$
 and satisfies the estimate
\begin{align}\label{followest}
\sup_{\tau \in [0,t]} ||y(\tau)||_Y \le& ~  ||y_0||_Y + k_T ||\overline w||_{L^2(0,t;H^{2}_*(0,L))},~~~\forall\, t\in [0,T].
\end{align}
\end{theorem}

The above theorem emphasizes the fact that the perturbation $\bP$ acting outside of $Y$ is regularized when incorporated into the operator $L$ in \eqref{mild-sol}; namely the operator $L$ is a priori  continuous from $L^2(0,T;H^2_*)$ to $C(0,T;[\cD(\bA ^*)]')$, however the ``hidden'' regularity of the trace of $\psi$ for solutions to \eqref{flow1} allows us to bootstrap $L$ to be continuous from $L^2(0,T;H^2_*)$ to $C(0,T;Y)$, with the corresponding estimate. We will invoke the results specifically presented in \cite[pp.645--653]{redbook} for Abstract Semigroup Convolution. For completeness, we include an outline below---for the complete proof, see \cite[Section 4.2]{KJ}.
\begin{proof}[Proof Outline] 

\noindent {\bf Step 1:} The operator $\bP_{\#}$, as defined by \eqref{op-P}, satisfies the following inequality for $\la\in\R$, $\la\neq 0$:
  \begin{equation}\label{y-to-z2}
   ||(\la-\bA)^{-1}\bP_{\#}( w)||_{Y}\le C_{U,\la} \|w\|_{H^2(\Om)}, ~~~\forall\, u\in H^2_*(\Om).
 \end{equation}
Here, $(\la-\bA)^{-1}\,  : [\cD(\bA)]'\mapsto Y$ extends to a mapping  from $Y$ to  $[\cD(\bA)]'$. This admits the abstract semigroup convolution framework.
\vskip.1cm
\noindent {\bf Step 2:} We  then obtain the following proposition. 
\begin{proposition} \label{weak-eq}
Let $\overline{w} \in C^1(0,T;H^{2}_*)$ and $y_0\in Y$. Then $y(t)$
given by \eqref{solution1} belongs to  $C([0,T];Y)$ and
is a strong solution to \eqref{inhomcauchy} in $[\cD(\bA)]'$, i.e.
in addition we have that
\[
y\in  C^1((0,T);[\cD(\bA)]')
\]
and  \eqref{inhomcauchy} holds in $[\cD(\bA)]'$ for each $t\in (0,T)$.
\end{proposition}
\noindent {\bf Step 3:} Let $y_T=e^{\bA^*T}y=e^{-\bA T}y$ (since $\bA$ is a skew-adjoint). Then obtain an estimate of the form
 \begin{align*}
||\bP^*_\#e^{\bA t}y_T||^2_{L^2(0,T:H^{-2}(0,L))} \le C_{T}\|y\|^2_Y.
\end{align*}
The estimate \eqref{trace-reg-est-M0} is crucial in proving the inequality of this step.
\vskip.2cm
\noindent {\bf Step 4:} The convolution framework then yields the estimate \eqref{followest}.
\end{proof}

Now, let  $\bX_t = C\big([0,t];Y\big).$
Now, take $\overline{y}=(\overline{\phi},\overline{\psi};\overline{w},\overline{v}) \in \bX_t$ and $y_0 \in Y$, and introduce the map $\cF: \overline{y} \to y$ given by
\begin{equation*}
y(t) = e^{\bA t}y_0+L(\overline w)(t),
\end{equation*}
i.e. $y$ solves \begin{equation*}%\label{absproblem2}
y_t=\bA y+\bP_\# \overline w,
~~ y(0)=y_0,
\end{equation*}
in the semigroup sense, where $\bP_\#$ is defined in \eqref{op-P}.
 It follows from (\ref{followest})
 that for $\overline y_1, \overline y_2 \in \bX_t$
 \begin{align*}
 \|\cF \overline y_1 - \cF \overline y_1\|_{\bX_t}\le & ~k_T ||\overline w_1 - \overline w_2||_{L^2(0,t;H^{2}_*(0,L))} \\
 \le & ~k_T \sqrt{t}\sup_{\tau \in [0,t]}|| \overline w_1 - \overline w_2||_{H^2(0,L)}
 \le k_T \sqrt{t} ||\overline y_1 - \overline y_2 ||_{\bX_t}.
 \end{align*}
Hence there is $0<t_*<T$ and $q<1$ such that
\[
\|\cF \overline y_1 - \cF \overline y_2\|_{\bX_t}\le q \| \overline y_1 -  \overline y_2\|_{\bX_t}
\]
for every $t\in [0,t_*]$.
This implies that  on the interval $[0,t_*]$ the problem
\begin{equation*}%\label{absproblem-ff}
y_t=\bA y+\bP y, ~~t>0,~
~~ y(0)=y_0,
\end{equation*}
has a local-in-time semigroup solution defined  in $Y$. This above local solution
can be extended to a global solution in finitely many intervals via the linearity of the problem.
Thus there exists a unique function
$y=(\phi,\psi;w,v)\in C\big([0,T];Y\big)$ for any $T>0$ such that
\begin{equation}\label{eq-fin}
y(t)=e^{\bA t}y_0+\int_0^te^{\bA(t-s)}\bP (y(s))ds~~\mbox{ in }~ Y
~~\mbox{for all }~t>0.
\end{equation}
It also follows from the analysis above that
\[
\|y(t)\|_Y\le C_T \|y_0\|_Y,~~~ t\in [0,T],~~\forall\, T>0.
\]
 This unique solution $y(t)$ \eqref{eq-fin} a posteriori gives the strongly continuous ``solution" semigroup 
$\widehat{T}(t)$ on $Y$ \cite{pazy}. It is then standard \cite{supersonic,KJ} to identify
the generator $\widehat{\bA}$ of $\widehat{T}(t)$  with the form
\begin{equation}\label{def-bA-n}
\widehat{\bA}z=\bA z+\bP z,~~z\in\cD(\widehat{\bA})=\left\{z\in Y\,:\; \bA z+\bP z\in Y\right\}.
\end{equation}
Hence, the semigroup trajectory $e^{\widehat \bA t}y_0$ is a semigroup (mild) solution for $y_0 \in Y$ (resp. a classical solution for $y_0 \in \cD(\widehat \bA)$) to \eqref{flowplate} with $\sigma=1$ on $[0,T]$ for all $T>0$. Moreover, we can write
\begin{equation}\label{dom-bA-n}
\cD(\bA + \bP) \equiv \left\{ y \in Y\; \left| \begin{array}{l}
-U \Dx \phi + \psi \in H^1(\R^2_+),~\\ -U \Dx \psi  -\bA_0 (\phi - N (v + U \Dx w ))+N(v+U\Dx w) \in L^2(\R^2_+) \\
v \in \cD(\cA^{1/2} )= H_*^2,~
-\cA w + N^* (\bA_0^*+I) \psi \in L^2(0,L) \end{array} \right. \right\}
\end{equation}
A posteriori we note that $\mathscr D(\mathbb A)$ and $\mathscr D(\widehat{\mathbb A})$ are topologically equivalent, and hence choosing data in $\mathscr D(\mathbb A)$ is sufficient to produce strong solutions for the full Cauchy problem \eqref{flowplate0}. 

{\bf Summarizing}: We consider $\widehat {\mathbb  A}$ as a semigroup generator for \eqref{flowplate} with $\sigma=1$ on $Y$, with $\mu>0$, with $\mathscr D(\widehat \bA) = \mathscr D(\bA)$. This is to say, for $y_0 \in \mathscr D(\bA)$ we have {\em strong solutions} to \eqref{flowplate} in the sense of Section \ref{solutions}.

\section{Nonlinear System} \label{Section3}
With the linear theory of strong solutions  now established (via the agency of the semigroup generator  $\widehat {\mathbb  A}$), we can incorporate the physically-derived, nonlinear stiffness effects associated with the large deflections of cantilevered beams. We recall the structural nonlinearity, as derived in \cite{inext1,maria2}, taking the form:
$$F(w) =- D\partial_x\big[(w_{xx})^2w_x \big]+D\partial_{xx}\big[w_{xx}(w_x)^2\big] =D \left [ w_{xx}^3+4w_xw_{xx}w_{xxx}+ w_x^2 \partial^4_x w \right].$$
The nonlinearity is quasilinear in nature, and cannot straightforwardly be treated as a perturbation.

To date, flow-beam and flow-plate interaction studies have utilized structural nonlinearities corresponding to restoring forces which are locally Lipschitz on the finite energy space---those coming from the von Karman (or related) theories \cite{survey1,supersonic,book}. The quasilinear nature of the present nonlinearity, however, is particularly challenging to incorporate into a semigroup framework, {\em hence marking a fundamental departure from the standard analysis of flow-plate interactions}. We will follow a multi-step approach,  briefly outlined below. As pointed out in the previous section, the domains $\mathscr D(\widehat{\mathbb A})$ and $\mathscr D(\mathbb A)$ coincide, so for simplicity we  henceforth just write $\mathscr D(\mathbb A)$.

\begin{enumerate}
\item We begin with the linear system $y_t = \widehat{\bA}y$, with initial data $y_0 \in \mathscr D(\bA)$. 
Since the beam nonlinearity $F(w)$ is not an energy-level perturbation, we introduce strong structural damping of Kelvin-Voigt type to the beam equation; the beam equations sees the addition of the ``regularizing" term $\delta \partial_x^4 w_t$. 
This produces the damped linear system $y_t = \widehat{\bA}_\delta y$. 
A key feature (see \cite{chen-triggiani}) is that $\mathscr D(\bA) $ is strictly contained in $\mathscr D(\bA_\delta)$. 
In this  framework, the nonlinearity $F(\tilde w)$ (for $\tilde w$ given) can be incorporated as a perturbation via the variation of parameters formula, for the system
$y_t = \widehat{\bA}_\delta y + \mathcal F(\tilde y).$

\item We then construct a fixed point in a particular subset of $L^2(0,T_\delta;\mathscr D(\bA_\delta))$, with the initial data taken from the smaller space $\mathscr D(\bA)$. 
To control the nonlinear terms, we impose size restrictions on the beam variables, encoded into the fixed point space. 
The strong damping with $\delta > 0$ plays a critical role, ensuring that the relevant estimates close to give that the fixed point mapping $\tilde y \mapsto y$ is well-defined and contractive. 
The construction of the solution, including the time of existence, $T_\delta$, depend explicitly on the regularization parameter $\delta$. 
Consequently, the fixed point provides a local-in-time strong solution on $(0,T_\delta)$.

\item Next, we track the dependence on $\delta$. 
We insert $y_\delta$ into the regularized system and re-establish energy estimates, now in the form of a priori bounds on smooth solutions. 
This step is delicate: the natural beam test functions used for the $\delta$-problem interact nontrivially with the coupling to the flow dynamics. 
To overcome this, we harness the power of the already established linear semigroup, differentiating the coupled system and carefully estimating higher time derivatives of the beam. 
In this framework, cancellations appear (as in the pure beam case \cite{maria2} and in aeroelastic problems with traces), allowing us to obtain good a priori estimates in $\mathscr D(\bA_\delta)$ {\em which are uniform in $\delta$.} 
Since $\mathscr D(\bA_\delta)$ itself does not directly control higher spatial derivatives, we invoke an equipartition argument (as for second-order systems) to subsequently recover those. 
The positivity/variational structure of the beam nonlinearity term is essential here. 
The result is $L^2$-in-time control of higher derivatives, and a priori bounds independent of $\delta \ge 0$.

\item Finally, assuming higher regularity of the initial data in $\mathscr D(\bA)$, we are able to pass to the limit as $\delta \searrow 0$, using the uniform estimates in the previous step. 
This yields a local strong solution to the undamped flow--beam system, first in the restricted $L^2(0,T;\mathscr D(\bA_\delta))$ setting. 
The restriction can then be lifted, giving a solution to the original system without any imposed damping. 
The final is a local strong solution to the full nonlinear flow-beam system, captured in Theorem \ref{th:main2}.
\end{enumerate}

\subsection{Preliminary Estimates for the Independent Beam}\label{beamests}

This section is devoted to the presentation of the critical estimates in the case of an uncoupled, nonlinear beam, driven by a generic force $p$. The details of this analysis come from \cite{maria2}, which deals precisely with the well-posedness of strong beam solutions. The work there provides a  road map for estimating the coupled flow-structure dynamics. Of particular relevance is the phenomenon that, while a priori bounds yield good estimates due to cancellation of the critical nonlinear terms, the actual construction of solutions is problematic; the latter requires consideration of a more regular dynamics with an incremental regularization parameter. Showing that  regularized solutions satisfy a priori estimates which are independent of the regularization parameter is the main thrust of our overall approach. 

We present the key points of the uncoupled beam analysis for self-containedness, but point to \cite[Section 4]{maria2} for additional details.  Recall the beam operator $\mathscr A$ with $\mathscr D(\cA) = \{ w \in H^4(0,L) : w(0) =w_x(0)=0; \; w_{xx}(L)=w_{xxx}(L) = 0\}.$
Consider the nonlinear beam system in \eqref{dowellnon}.
We utilize several notions of beam energies below. The baseline energy is given by:
\begin{equation}\label{energiesdef}
E_0(t) \equiv  \frac{1}{2}\left[||w_t||^2\right]+\dfrac{D}{2}\left[||w_{xx}||^2+ ||w_xw_{xx}||^2\right],
\end{equation}
and is useful for stating the main existence theorem for the uncoupled beam.
\begin{theorem}\label{withoutiota}
Let $p \in H^1((0,T) \times (0,L))$ for any $T>0$. For initial data $(w_0,w_1) \in \mathscr D(\mathscr A)\times H_*^2$, strong beam solutions exist up to some time $T^*(w_0,w_1,p)$. For all $t \in [0,T^*)$, the solution $w$ is unique and obeys the energy identity
$$E(t) = E(0)+\int_0^t (p,w_t)_{L^2(0,L)}d\tau.$$ 
\end{theorem}
\begin{remark}\label{rem:T*}
The time of existence $T^*$ depends only on the {\em size of the data} in the sense of
\[
T^* = T^*\big(\|(w_0,w_1)\|_{\mathcal{D}(\mathscr{A})}, \|p\|_{H^1(0,T;L^2(0,L))}\big).
\]
The results in \cite{maria2} require more stringent hypotheses, but they also provide a notion of continuous dependence in $\mathcal{D}(\mathscr{A})$ which is not central here.
\end{remark}

The first energy inequality one can obtain for strong solutions  is:
\begin{equation}
\label{FirstLevel}
E_{0}(t) \leq f_0 \left( ||p||_{L^2\left(0,t;L^2(0,L)\right)}, ||w_0||_{\mathscr D(\mathscr A)},||w_1||_{H^2_{*}} \right) e^{t/2} ~~\text{for all}~~ t>0.
\end{equation}
The function $f_0$ is increasing in its arguments.
Since we assume $p,p_t \in L^2(0,T;L^2(0,L))$, $||p(0)||$ can be interpreted as a temporal trace, with $||p(0)|| \lesssim ||p||_{H^1(0,T;L^2(0,L))}$ and this yields:
\begin{equation}
\label{boundednessofwtt}
||w_{tt}(0)|| \leq f \left( ||p||_{H^1(0,T;L^2(0,L))}, ||w_0||_{\mathscr D(\mathscr A)} \right).
\end{equation}
The higher-order energy identity we utilize is given by:
\begin{align}\label{energy1}
\frac{1}{2} \frac{d}{dt}\big[  ||w_{tt}||^2 +  D||w_{xxt}||^2 &+D ||w_{xx}w_{xt}||^2 +  D ||w_{xxt}w_{x}||^2\big]  \\ \nonumber
=-\dfrac{d}{dt}\Big[4D(w_{x} w_{xx} , & w_{xt}w_{xxt})\Big]+\left( p_{t}, w_{tt} \right) + 3D(w_{xx} w_{xxt} ,w^2_{xt} ) +3 D ( w_{x}w_{xt},w^2_{xxt}) .
\end{align}
Invoking the higher energy definition:
\begin{equation}
\label{secondenerg}
E_{1}(t) = \frac{1}{2} \left[ ||w_{tt}||^2 +  D||w_{xxt}||^2 + D||w_{xt}w_{xx}||^2 + D||w_{x}w_{xxt}||^2 \right],
\end{equation} 
we can then obtain:
\begin{align}
\label{timediff}
E_1(t) \leq& ~ f_1\left(||p_t||_{L^2(0,t;L^2(0,L))}, ||w_0||_{\mathscr D(\mathscr A)}, ||w_1||_{H^2_{*}} \right)  \\ & ~+ f_2\left(||p||_{L^2(0,t;L^2(0,L))},||w_0||_{\mathscr D(\mathscr A)},||w_1||_{H^2_{*}} \right)t  + C \int_0^t \left [ E_1(\tau) \right ] ^2d\tau.\nonumber 
\end{align}
From this follows the central local-in-time estimate:
\begin{equation}
\label{NonlinGronw}
E_1(t) \leq  \frac{ f_1  + f_2t }{1- C \left[ f_1 t + f_2 t^2 \right]} \equiv M_1(t),~ ~~0 \leq t < T^* ~~\text{where}~~ T^* = \sup_{t >0} \left \{ C \left[ f_1 t + f_2 t^2 \right] <1 \right \}. 
\end{equation}
Then, for any fixed $T<T^*$, we have that \eqref{NonlinGronw} constitutes a priori bound on $ E_1(t)<M_1^*(T),~~t \in [0,T]$, where \begin{equation}M_1^*(T) = \max_{t \in [0,T]} M_1(t);\end{equation} this quantity depends only on fixed norms of the data and $T$. 

Finally, we can use the nonlinear beam equation  itself in \eqref{dowellnon} to bound $w$ in $\mathscr D(\mathscr A)$:
\begin{equation}\label{lastlast} ||(1+w_x^2)\partial^4_x w|| \lesssim ||w_{tt}||+||w_{xx}||_{L^{\infty}(0,L)}^2||w_{xx}|| +||w_x||_{L^{\infty}(0,L)}||||w_{xx}||_{L^{\infty}(0,L)}||||w_{xxx}||+||p||\end{equation}
where norms without subscripts are taken to be $L^2(0,L)$. From which it follows that the strong solution has \begin{equation}\label{finaest}(w,w_t,w_{tt}) \in L^{\infty}(0,t;\mathscr D(\mathscr A) \times H^2_* \times L^2(0,L))~\text{ for all }~0 \le t <T^*.\end{equation}

\subsection{Fixed Point Framework for Damped Flow-Plate System}
In line with the guiding idea from the previous section, we first define the regularized model for the full flow-structure interaction. Specifically, we consider \eqref{fullpot*} with $\beta=1$, where the beam equation is augmented with strong damping and modified to:
\begin{equation}  
w_{tt}+ \partial^4_x [Dw+\delta w_t] - D\partial_x\big[(w_{xx})^2w_x \big] +D \partial_{xx}\big[w_{xx}(w_x)^2\big]= r_{(0,L)}\gamma_0  [\phi_t+U\phi_x] . 
\end{equation}

Then, the full system with the  nonlinear beam model takes the form:
\begin{equation}\label{beamFP}
\begin{cases}
(\partial_t + U\partial_x)^2\phi=\Delta_{x,z} \phi & \text{ in } \mathbb{R}_+^2\times (0,T)\\
\partial_z\phi =  (\partial_t+U\partial_x)w & \text { on } (0,L)\\
(\phi_t+U\phi_x) = 0  &\text{ on } (-\infty,0) \cup (L,\infty)\\
\phi(0)=\phi_0,~~  \phi_t(0)=\phi_1
\\
 w_{tt} + D\partial_x^4 w + \delta \partial_x^4 w_t = -F(w)+ r_{(0,L)}\gamma_0  \phi_t+U\phi_x]& \text{ in } (0,L) \times (0,T)\\
w=w_x = 0 & \text { in }\{0\}\times (0,T)\\
w_{xx}=0;~~w_{xxx}=0 & \text{ in }\{L\} \times (0,T)\\
w(0)=w_0,~~  w_t(0)=w_1 \end{cases}
\end{equation}

This modification corresponds to a flow--plate generator $\widehat{\mathbb A}_{\delta}$, adapted from \eqref{def-bA-n}--\eqref{dom-bA-n}, which incorporates strong damping in the beam equation. An important difference, however, is that the structurally damped operator defines a larger domain than in the undamped case. In particular, domain membership requires only $
\partial_x^4 (Dw+\delta w_t) \in L^2(0,L),$
which does not in itself provide direct information about the regularity of $\partial_x^4 w$. In fact, we only have that $w_{xxt} \in L^2(0,T)$. For this reason, special care is required when defining the set on which the construction of solutions will take place.

For the fixed point setup, we replace (linearize) $F(w)$ with $F(\tilde w)$, where $\tilde w$ is appropriately specified. The initial data $y_0=(\phi_0,\psi_0;w_0,w_1) \in \mathscr D(\bA_{\delta})$ are considered fixed. The regularity of the flow  $(\phi, \psi)$ in the domain does not depend on $\delta$, since 
$
\gamma_0[\phi_t+U\phi_x]_{(0,L)} \in C([0,T]; L^2(0,L)).
$
Thus, the problem admits well-defined linear regularized semigroup dynamics, which provides a natural basis for setting up the fixed point argument for the nonlinear system.
Abstractly, the system nonlinearity then takes the form $
\mathcal{F}(\tilde{{y}}) = (0, 0; 0, -F(\tilde w)),
$ 
and we write the linearized system for strong solutions as:
\begin{equation} \label{FP}
{{y}_t} = \widehat{\bA}_\delta {y} + \mathcal{F}(\tilde{{y}}), \quad {y}(0) = {y}_0 \in \mathscr D(\bA_{\delta}),
\end{equation}
where \({y} = (\phi, \psi; w, v)\), and \(\tilde{{y}} = (\tilde{\phi}, \tilde{\psi}; \tilde{w}, \tilde{v})\). Given $\tilde y$, we solve for $y$ and 
this defines the fixed point map, for which we will invoke the Banach contraction principle on an appropriately defined subset of data in $\mathscr D(\bA_{\delta})$. 
Accordingly, consider
\begin{equation}\label{Tmap}
\mathcal{T} : \tilde{{y}} \mapsto {y},
\end{equation} where $\mathcal T(\tilde y) = y$,
given through the variation of parameters mapping
\begin{equation}\label{varpar2}
{y}(t) = e^{\widehat{\bA}_\delta t} {y}_0 + \int_0^t e^{\widehat{\bA}_\delta(t-s)} \mathcal{F}(\tilde{{y}}(s)) \, ds.\end{equation}
With the semigroup generated by $\widehat{\bA}_{\delta}$, the map $\mathcal{T}$ is well-defined (and produces a strong semigroup solution) whenever $\tilde y \in C([0,T];\mathscr D(\bA_{\delta}))$ \cite{pazy}. Since both the initial data and the forcing term belong to the domain of the generator, it follows that ${y}(t) \in C([0,T]; \mathscr D({\bA}_\delta))$ for all $t$.

Our goal is to prove that $\mathcal{T}$ admits a unique fixed point in a suitable closed set $\mathcal{U}\subset \mathscr D({\bA}_\delta)$via the contraction mapping principle on a closed, dynamics-invariant ball of strong solutions, measured in the topology of weak solutions (the so-called finite energy space). Thus we work within the framework of a convex metric space. 
The approach here is two-fold: first, we show that $\mathcal{T}$ maps a ball in a stronger space into itself, and second, we establish the contraction property in a weaker, finite energy space, as is standard  standard for quasilinear PDEs.

\subsection{The Fixed Point Argument}

We define the following spaces that will be used to run the fixed point argument. First, consider for $y=(\phi,\psi;w,v) \in \mathscr D(\bA_{\delta})$:
\[
\mathcal{Z}_T \equiv \left\{ {y} \in C([0,T]; \mathscr D({\bA}_\delta)) \ :\  \int_0^T ||w(t)||_{\mathscr D(\mathscr A)}^2 dt +  \sup_{t \in [0,T]} \left(  \|v(t)\|_{H^2_*}^2 \right) < L \right\},
\]
where \(L > 0\) is a fixed constant. (Here we use Poincar\'e and the boundary conditions to identify the $H^4(0,L)$ on $\mathscr D(\mathscr A)$ with $||\partial_x^4\cdot||_{L^2(0,L)}$.) In addition, we define a weaker topology using the space:
\[
X_T \equiv \left\{ {y} \in C([0,T]; \mathscr D({\bA}_\delta)) \ :\ \sup_{t \in [0,T]} \left( \|w(t)\|_{H^2_*}^2 + \|v(t)\|_{L^2(0,L)}^2 \right) < 4\left (  \|w(0)\|_{H^2_*}^2 + \|v(0)\|_{L^2(0,L)}^2 \right ) \right\}.
\]
We emphasize that in both definitions, the constraints are imposed only on the beam variables \(w\) and \(v\), while the flow variables \((\phi, \psi)\) retain the regularity associated to their membership in \(\mathscr D({\bA}_\delta)\).

In what follows, we will demonstrate that the fixed point map \(\mathcal{T}\) maps  
\begin{equation}\label{definitionofU} 
\mathcal{U} \equiv X_T \cap \mathcal{Z}_T
\end{equation}
into itself for $T$ sufficiently small, and is a contraction with respect to the weaker topology induced by \(X_T\). While the contraction argument is carried out in the space \(X_T\), the additional regularity imposed by \(\mathcal{Z}_T\) ensures that the nonlinear forcing term is well-defined and allows us to derive the necessary a priori estimates.

For clarity, we note that the definitions of $\mathcal{Z}_T$ and $X_T$ above depend on the damping parameter $\delta$ through the underlying domain of the generator for $\widehat{\bA}_{\delta}$. Hence we should write $\mathcal{Z}_{T,\delta}$ and $X_{T,\delta}$. Accordingly, we will denote $\mathcal{U}_\delta \equiv X_{T,\delta} \cap \mathcal{Z}_{T,\delta}$.

\begin{theorem}[Unique Fixed Point for  $\delta$-problem]\label{FPExistence} 
Let $\delta>0$. There exists a time $T_\delta > 0$, sufficiently small, such that, for 
$\mathcal{U}_\delta$, the mapping
\[
\mathcal{T} : \mathcal{U}_\delta \to \mathcal{U}_\delta
\]
admits a unique fixed point. In other words, there exists a unique 
$y_\delta \in \mathcal{U}_\delta$ such that $\mathcal{T}(y_\delta) = y_\delta$.
\end{theorem}

Later we will show that if the initial condition satisfies $y_0 \in \mathscr D({\bA}) \subset \mathscr D({\bA}_\delta)$, then both $T_\delta$ and $\|y\|_{\mathcal{Z}_{T}}$ can be chosen independent of $\delta$.

\begin{proof} 

\noindent {\bf{Step 1: Self-mapping}.} We aim to show that for any fixed $\tilde{{y}} \in \mathcal{U}_{\delta}$, the solution ${y} = \mathcal{T}(\tilde{{y}})$ remains in $\mathcal{U}_{\delta}$ for a time \(T\), sufficiently small, which will depend on the parameter $\delta$. Since the nonlinearity appears only in the beam component of the system, our analysis here focuses on that equation. The flow system is linear, so for given $(w,w_t)$, it can be estimated directly; indeed, we  appeal directly to the identity \eqref{energyrelation*} and the flow-trace regularity in \eqref{trace-reg-est-M0} which applies for a flow solution $y=(\phi,\psi;w,v)=\mathcal T(\tilde y)$. Namely, for the solution we have: 
\begin{align}\label{f1}
\sup_{t \in [0,T]}& \Big( \|\psi(t)\|_{L^2}^2 
    + \|\nabla_{x,z} \phi(t)\|_{L^2}^2 
    + \mu \|\phi(t)\|_{L^2}^2 
    + E_0(t) \Big) \notag \\
    \leq &  C \Big( \|\psi_0\|_{L^2}^2 
    + \|\nabla_{x,z} \phi_0\|_{L^2}^2 
    + \mu \|\phi_0\|_{L^2}^2 
    + E_0(0) \Big) 
   + |U| \int_0^T |(w_x, r_{(0,L)}\gamma_0[\psi])_{L^2(0,L)}| \, d\tau. 
\end{align}
where the beam energy is given by \eqref{energiesdef}.

It now remains to obtain suitable bounds for the beam variables $(w,w_t)$. To this end, we first derive an estimate in the lower norm by testing the beam equation \eqref{beamFP} with \(w_t\) and integrating over the domain to obtain the identity:
\[
\frac{d}{dt} E_0(t) + \delta \|w_{xxt}(t)\|^2 = ( -F(\tilde w), w_t ) + (r_{(0,L)} \gamma_0[\psi], w_t ).
\]

We estimate the nonlinear and coupling terms using Sobolev embeddings and Young's inequality:
\begin{equation}\label{estimates}
|(F(\tilde w), w_t)| \leq \epsilon \|w_{xxt}\|^2 + C_\epsilon \|\tilde w\|_{\mathscr D(\mathscr A)}^{2k}, ~~
|(r_{(0,L)}\gamma_0[\psi], w_t)| \leq \epsilon \|w_{xxt}\|^2 + C_\epsilon \|r_{(0,L)}\gamma_0[\psi]\|_{L^2(0,L)}^2,
\end{equation}
where $k$ is a positive integer and the constants $C_\epsilon$ depend on $\delta$.

Here, the trace regularity is in fact \(\gamma_0[\psi] \in H^{1/2}(\mathbb R)\), following from the fact that, since ${y} \in C([0,T]; \mathscr D({\bA}_\delta))$, we have  \(\psi \in C([0,T]; H^1(\mathbb R^2_+))\) from Theorem~\eqref{domainchar}. Then, integrating in time and combining estimates, and selecting $\epsilon$ small relative to $\delta>0$ fixed, we obtain (after dropping the damping term):
\[
E_0(t) \leq E_0(0) + C_{\delta} T L^{2k} + C_{\delta} T \sup_{t \in [0,T]} \|\psi(t)\|_{H^{1}(\mathbb R^2_+)}^2.
\]
Thus, there exists \(T^{\#}_{\delta}>0 \) small enough so that: $\ds
\sup_{t \in [0,T^{\#}_{\delta}]} E_0(t) \leq 2E_0(0).$

We now pursue an estimate in the higher norm. To this end, we test equation \eqref{beamFP} with $\partial_x^4 w_t$ and define the higher-order energy:
\[
E_2(t) \equiv \frac{1}{2} \left( \|w_t(t)\|_{H^2_*}^2 + \|w(t)\|_{\mathscr D(\mathscr A)}^2 \right) 
= \frac{1}{2} \left( \|\partial_x^2 w_t(t)\|^2 + \|\partial_x^4 w(t)\|^2 \right).
\]
This yields the identity:
\[
\frac{d}{dt} E_2(t) + \delta \|\partial_x^4 w_t(t)\|^2 
= -( F(\tilde w), \partial_x^4 w_t ) + (r_{(0,L)} \gamma_0[\psi], \partial_x^4 w_t ).
\]
Using the strength of the damping and the definition of $\mathcal Z_T$, we estimate as in \eqref{estimates} to obtain:
\[
\sup_{t \in [0,T]} E_2(t) \leq E_2(0) + C_{\delta} T L^{2k} 
+ C_{\delta} T \sup_{t \in [0,T]} \|\psi(t)\|_{H^{1}(\mathbb R^2_+)}^2.
\]
Thus, for sufficiently small $T = T^\#_\delta$ (understood as the minimum of the restrictions obtained from $E_0$ and $E_2$ estimates), the beam components $(w,w_t)$ remain within the ball 
of radius $L$ in the norm of $\mathcal{Z}_{T,\delta}$.
By combining the $E_0$ and $E_2$ estimates for the beam component with the flow bounds in \eqref{f1}, we conclude that, for sufficiently small $T^\#_\delta$, the fixed point map $\mathcal{T}$ maps 
$\mathcal{U}_{\delta}$ into itself. The size of this set depends on $\delta$.

\medskip
\noindent {\bf{Step 2: Contraction Property}.} Let $\tilde{{y}}_1, \tilde{{y}}_2 \in \mathcal{U}_{\delta},$ and define the images under the fixed point map as \( {y}_j \equiv \mathcal{T}(\tilde{{y}}_j) = (\phi_j, \psi_j; w_j, v_j) \), for \( j = 1, 2 \). Both trajectories $\tilde y_1$ and $\tilde y_2$ have the same initial conditions (by definition of the fixed point mapping). Secondly, since the flow equation is linear, estimating differences thereof is handled straightforwardly (with zero initial data, by the principle of superposition). Thus we follow the structure of the argument in in the linear case in Step 4 of Section \ref{fulldyn}.

We now focus on the estimates for $\mathcal{T}(\tilde{{y}}_1) - \mathcal{T}(\tilde{{y}}_2)$ which will come from the beam variables in the difference \( z = w_1 - w_2 \).
Specifically, $z$ satisfies the beam equation:
\[
z_{tt} + \partial_x^4 z + \delta \partial_x^4 z_t = -[F(\tilde w_1) - F(\tilde w_2)]+r_{(0,L)}\gamma_0[\psi_1-\psi_2], \quad z(0) = 0, \quad z_t(0) = 0.
\]
The flow equation remains the same as in \eqref{fullpot*}, by linearity, and thus we can invoke the energy relation akin to \eqref{energyrelation*} by testing the beam with $z_t$ and the flow $\psi_1-\psi_2$.
Denoting
\begin{align}\label{energies**}
~~\mathcal E_z(t) =  \dfrac{1}{2}\big(||z_t||^2+D||z_{xx}||^2\big)+\dfrac{1}{2}\big(||\psi_1-\psi_2||^2+||\nabla_{x,z} [\phi_1-\phi_2]||^2+\mu||\phi_1-\phi_2||^2\big)
\end{align}
we have (noting that $\mathcal E_z(0)=0$):
\begin{equation}\label{energyrelation**}
\mathcal E_z(t)+\delta\int_0^t||z_{xxt}||^2d\tau + U\int_0^t ( z_x,r_{(0,L)}\gamma_0[\psi_1-\psi_2])_{L^2(0,L)} dt =-\int_0^t(F(\tilde w_1) - F(\tilde w_2), z_t)d\tau. \end{equation}

The nonlinear terms are estimated using their polynomial nature and standard Sobolev embeddings, as in \cite{maria2}:
\begin{equation}
|(F(\tilde w_1) - F(\tilde w_2), z_t)| \leq C L^k \|\tilde w_1 - \tilde w_2\|_{H^2_*} \|z_t\|_{L^2(0,L)}
\leq \epsilon \|z_t\|_{L^2(0,L)}^2 + C_\epsilon L^{2k} \|\tilde w_1 - \tilde w_2\|_{H^2_*}^2.
\end{equation}
We may use the above in \eqref{energyrelation**} to absorb \( \|z_t\|_{L^2}^2 \) via Poincar\'e with $\epsilon$ chosen small with respect to the fixed $\delta$.
This produces the intermediate inequality:
\begin{align}
\mathcal E_z(t) \le  &~ C L^{2k} \int_0^T\|\tilde w_1 - \tilde w_2\|_{H^2_*}^2d\tau + |U|\int_0^T \left|( z_x,r_{(0,L)}\gamma_0[\psi_1-\psi_2])_{L^2(0,L)}\right| dt.
\end{align}
Then, we estimate as:
$$\big|( z_x,r_{(0,L)}\gamma_0[\psi_1-\psi_2])_{L^2(0,L)}\big| \le ||z_x||_{H^1_*(0,L)}||r_{(0,L)}\gamma_0[\psi_1-\psi_2])_{L^2(0,L)}||_{H_*^{-1}(0,L)}$$
and use \eqref{trace-reg-est-M0} adapted here with $\epsilon=1/2$:
\begin{equation}\label{trace-reg-est-M00}
\int_0^T\|r_{(0,L)}\gamma_0[\psi](t)\|^2_{H_*^{-1} (0,L)}dt\le c\left(
 \int_0^T\| z_t+Uz_x \|_{L^2(0,L)}^2dt\right) \le c\int_0^T\mathcal E_z(t)dt.
\end{equation}
and the zero initial data of the differences for $\phi_1-\phi_2$ and $\psi_1-\psi_2$.
We then have:
\begin{equation} \mathcal E_z(t) \le  ~C L^{2k}T \sup_{t \in [0,T]} \|\tilde y_1-\tilde y_2\|^2_Y+c|U|\int_0^T\mathcal E_z(t)dt, \end{equation}
where $Y$ as in \eqref{statespace}. From here, we can invoke Gronwall's inequality, to obtain:
\begin{equation} \mathcal E_z(t) \le  ~Ce^{c|U|T} L^{2k}T \sup_{t \in [0,T]} \|\tilde y_1-\tilde y_2\|^2_Y \le q(T) \sup_{t \in [0,T]} \|\tilde y_1-\tilde y_2\|^2_Y \end{equation}

By the intermediate value property, a $T^{\#}_{\delta}>0$ can be chosen sufficiently small so that $q(T^{\#}_{\delta})<1$. 
Therefore, after possible minorization of $T^{\#}_{\delta}$, and taking the supremum over $t \in [0,T^{\#}_{\delta}]$, the mapping $\mathcal{T}$ satisfies the contraction property:
\[
\|\mathcal{T}(\tilde{y}_1) - \mathcal{T}(\tilde{y}_2)\|_{X_T}^2
\leq q \|\tilde{y}_1 - \tilde{y}_2\|_{X_T}^2,
\]
for all $T \leq T^{\#}_{\delta}$.

The contraction mapping principle thus applies to the mapping $\mathcal T:\mathcal U_{\delta} \to \mathcal U_{\delta}$.
\end{proof}

After applying the contraction mapping principle, the resulting fixed point yields a local-strong solution for the damped system and completes the proof of Theorem \ref{FPExistence}.

\begin{corollary}[Local Existence of Nonlinear Strong Solutions via Contraction Fixed Point]
Let $\delta>0$ and consider \eqref{beamFP}. With the definition of $\mathcal T:\mathcal U_{\delta} \to \mathcal U_{\delta}$ in \eqref{Tmap}, together with the variation of parameters formula \eqref{varpar2}, we obtain a strong solution in the sense of Section \ref{solutions}, namely, $y_{\delta}(t) = (\phi^{\delta}(t),\psi^{\delta}(t);w^{\delta}(t),v^{\delta}(t)) \in \mathcal U_{\delta}$. We have the identifications $\psi^{\delta}(t) = \phi^{\delta}_t(t)+U\phi^{\delta}_x(t)$ and $v^{\delta}(t)=w^{\delta}_t(t)$. This solution is local in the sense of Remark \ref{locality}.
\end{corollary}
We will produce estimates on the resulting strong solution which are  independent of $\delta$   in the next section.

\subsection{Uniformity in $\delta$ }
We now show that the local solutions constructed above (with damping) admit a time of existence and a priori bounds that are independent of $\delta>0$. 
This later allows passage to the limit $\delta \searrow 0$ in \eqref{beamFP}, to obtain strong (local) solutions to \eqref{fullpot*}. 
We will require a priori estimates for smooth solutions
$y_\delta \in C([0,T];\mathscr D(\bA)) \cap C^1([0,T];Y),$
 independent of $\delta$; we will first obtain uniformity in  the sense of
$C^1([0,T];Y) \cap L^2(0,T; \mathscr D(\mathbb A)),$
which suffice for the limit passage, and then we will subsequently upgrade the time regularity from $L^2(0,T)$ to $C([0,T])$.
 
Running the estimates on both flow and beam with the constructed solution $y_\delta$, we rely on the a priori results of Section \eqref{beamests}. Therein, we observe the beam admits  $\delta$-independent estimates. The main remaining issue is to control the flow-beam coupling at the boundary. The standard approach of utilizing the beat test function $\partial_x^4 w_t$ is problematic from two points of view: first, in \cite{maria2}, we note that this is a problematic multiplier for the analysis of the beam nonlinearity; secondly, it would produce the  term,
$\int_0^L  \langle \partial_x^4 w_t , r_{(0,L)}\gamma_0[\psi]\rangle$, which requires $\psi_t \in L^2(0,L)$, while only $\psi_t \in H^{-1/2}_*(0,L)$ is available. We therefore pursue a different approach, outlined below.

\begin{enumerate}[label=(\roman*)]
\item We repeat the estimates at the lower energy level for the full system. Here any unsigned contributions from the beam nonlinearity cancel out, yielding a conservative estimate for the smooth solution, depending only on the initial data measured in the finite energy space.

\item Next, we differentiate the system in time. The resulting nonlinear terms reorganize into energy-level contributions, so the estimate is independent of $\delta$, though now depends on higher norms of the initial data and on fourth-order derivatives of the beam. This reflects the fact that the strong damping does not provide maximal regularity for elements in $\mathscr D(\bA_\delta)$.  

\item Lastly, we perform an additional ``equipartition" estimate on the beam equation using the multiplier $\partial_x^4 w$. This recovers maximal $ w \in L^2(0,T;\mathscr D(\mathcal A)) $ regularity, though this holds only for data in the smaller space $\mathscr D({\bA})\subset \mathscr D({\bA}_\delta)$. In this step, the favorable sign of the critical nonlinearity plays a major role.  
\end{enumerate}

The following theorem, together with its proof, covers these points and establishes the desired $\delta$-independent estimates on strong solutions, which can be applied to the constructed local, strong solution, $(\phi^{\delta}, \psi^{\delta};w^{\delta}, v^{\delta})$.

\begin{theorem}[$\delta$-independence]
The local strong solutions $y_\delta$ constructed in Theorem \ref{FPExistence} admit  a time of existence $T>0$ and uniform bounds that are both independent of the damping parameter $\delta \geq 0$.
\end{theorem}

\begin{proof}
This proof of this theorem is just the verification of the three steps outlined above. For each step, we establish the corresponding energy estimate and show emphasize its uniformity in $\delta$. To that end, we begin by introducing the two classes of energies $E_0, E_{f,0}$ and $E_1, E_{f,1}$ together with some preliminary observations, which will serve as the foundation for  steps (i)--(iii). Consider:
\begin{align}
E_{0} & = \dfrac{1}{2}\big(||w_t||^2+D||w_{xx}||^2+D||w_xw_{xx}||^2\big), \\
E_{f,0} &=  \dfrac{1}{2}\big(||\psi||^2+||\nabla_{x,z} \phi||^2 + \mu ||\phi||^2\big), \\ 
E_{1} & =  \frac{1}{2} \left[ ||w_{tt}||^2 +  D||w_{xxt}||^2 + D||w_{xt}w_{xx}||^2 + D||w_{x}w_{xxt}||^2 \right], \\
E_{f,1} &=  \dfrac{1}{2}\big(||\psi_t||^2+||\nabla_{x,z} \phi_t||^2 + \mu ||\phi_t||^2\big),
\end{align}
and ~$\mathcal E_i(t)  =  E_{i}(t)+E_{f,i}(t),~i=0,1$, 
along with the identity:
\begin{align}\label{enest1}
\mathcal E_0(t)+ U\int_0^t  \big \langle w_x,r_{(0,L)}\gamma_0[\psi]\big \rangle_{(0,L)} d\tau = \mathcal E_0(0).
\end{align}
The pairing $\langle \cdot , \cdot \rangle$  will be a posteriori justified as an $L^2$ inner product for smooth solutions, and can be justified as an $H^1_* \times H^{-1}_*$ duality pairing at the baseline energy level. 
The higher order energy estimate for $\mathcal E_1$ will be polluted by nonlinear terms coming from \eqref{energy1} and will be discussed below.

The trace estimate in \eqref{trace-reg-est-M0} for the flow system \eqref{flow1}, being linear in beam variables, applies to provide the a priori flow trace estimates:
\begin{align} \label{tracee1}
\int_0^T\|r_{(0,L)}\gamma_0[\psi](t)\|^2_{H_*^{-1/2 -\epsilon} (0,L)}dt\le&~ C\left(
E_{fl,0}(0)+ \int_0^T\| w_t+Uw_x \|_{L^2(0,L)}^2dt\right)\\\label{tracee2}
\int_0^T\|r_{(0,L)}\gamma_0[\psi_t](t)\|^2_{H_*^{-1/2 -\epsilon} (0,L)}dt\le&~ C\left(
E_{fl,1}(0)+ \int_0^T\| w_{tt}+Uw_{xt} \|_{L^2(0,L)}^2dt\right).
\end{align}
We  recall  the beam equation is forced by the flow   $\gamma_0[ \psi(t)]$,  necessitating  estimates on: 
$$(D^jw, r_{(0,L)}\gamma_0[\psi] )_{L^2(0,L)},\quad D^j \in \{\partial_t, \partial_t^2, \partial_x^4\}.$$

We must justify the estimates in \eqref{FirstLevel}--\eqref{lastlast} with $p$ given as ~$p=r_{(0,L)}\gamma_0[\psi]$. We note that in the analysis that follows, the presence of $\delta$ will be ignored, to demonstrate that strong solutions $y_{\delta}(t) = (\phi^{\delta}(t),\psi^{\delta}(t);w^{\delta}(t),v^{\delta}(t)) \in C([0,T]; \mathscr D({\bA}_\delta)$ can in fact be estimated independent of $\delta \ge 0$. In particular, in Steps (i) and (ii) the $\delta$-terms produced on the left-hand side are signed and conservative,
and thus will be dropped.

\medskip

\noindent\textbf{Step (i): Lower Energy Estimates.} For this step, we note that we can combine \eqref{enest1} and \eqref{tracee1} as in \eqref{trace-reg-est-M00}, along with Gronwall, to provide the desired boundedness at the level of $\sup_{t \in [0,T^{\#}]}\mathcal E_0(t)$ which is independent of $\delta >0 $. This is due to   cancellations between  the terms ~~$-\left(  \partial_x \big[(w_{xx})^2 w_x\big], w_t \right)$ ~and ~~$\left(\partial_{xx}\big[w_{xx}(w_x)^2\big], w_t \right)$, which yields a conservative bound at the lower energy level.

\medskip
\noindent\textbf{Step (ii): Estimates for Time-Differentiated System.} 
Moving on, we must differentiate the entire system \eqref{beamFP} in time and utilize the test functions $(\psi_t,w_{tt})$. 
The linear part of the system reflects the same structure as the calculations at the level of $\mathcal E_0$. 
For the nonlinear contributions of the beam, one obtains terms such as 
$(w_x^2 \partial_x^2 w_t,\, \partial_x^2 w_{tt})$ and $( \partial_x^2 w_{xt},\, \partial_{x} w_{tt})$, which at first appear problematic. 
However, at this stage these contributions reorganize into energy-level terms, so they can be controlled directly through interpolation and Young's inequality. 
This shows that no damping-dependent terms are needed in the bound,  demonstrated in \eqref{energy1}.  

We then write:
\[
J(t) = -4D(w_xw_{xx},w_{xt}w_{xxt})\big|_{0}^t+3D\int_0^t(w_{xx}w_{xxt},w_{xt}^2)+(w_xw_{xt},w_{xxt}^2)\,d\tau.
\]
From this, we obtain:
\begin{align}
\mathcal E_1(t) +U\int_0^t \big \langle w_{xt},r_{(0,L)}\gamma_0[\psi_t]\big \rangle_{(0,L)}\, d\tau 
= \mathcal E_1(0) + J(t).
\end{align}
The nonlinear term $J(t)$ is estimated exactly as in \eqref{timediff} \cite{maria2}, and the trace term is estimated using \eqref{tracee2}. 
From this point, Gronwall can be directly invoked as in previous steps. 
This provides an a priori bound of the form:
\[
\sup_{t \in [0,T]} \left[\mathcal E_1(t) +\mathcal E_0(t)\right] \le K\left(T,\mathcal E_1(0), \mathcal E_0(0)\right).
\]

\medskip

\noindent\textbf{Step (iii): Recovery of Fourth-Order Regularity.} 
To recover the required fourth-order regularity for the beam, we now test the beam equation with $\partial_x^4 w$. Note that this estimate cannot simply be read off from the structurally damped beam, since it involves 
a mixing of time and spatial derivatives. 
The nonlinear term turns out not to be an obstacle, since it contributes with the correct sign as we can see from the equality:
\[
-(F(w), \partial_x^4 w) = \|w_x \partial_x^4 w\|^2 + \mathcal{R}(t),
\]
where $\mathcal{R}(t)$ denotes lower-order terms. 

After invoking the KJC condition as in \eqref{tracee1}, we obtain:
\begin{align}
\frac{\delta}{2} &\|\partial_x^4 w(T)\|^2 
+ D\int_0^T \|\partial_x^4 w\|^2 \,dt
+ D \int_0^T ||w_x \partial_x^4 w||^2 \,dt \label{eq:4thorder}\\
& = \int_0^T (w_{tt}, \partial_x^4 w) \,dt
+ C \int_0^T \mathcal{R}(t)\,dt  \nonumber 
+ \frac{\delta}{2} \|\partial_x^4 w(0)\|^2 + \int_0^T (r_{(0,L)}\gamma_0[\psi], \partial_x^4 w) \,dt  .
\end{align}

The first two terms on the right-hand side are bounded by the initial data 
$\mathcal E_0(0),\mathcal E_1(0)$ (see \cite[Step 5, Section 4.3]{maria2}). 
In addition, the $\delta$-terms in \eqref{eq:4thorder} appear with the correct sign and can be bounded by the initial data; hence they play no essential role. 
Thus, the only remaining step is to obtain an estimate for $\|\gamma_0[\psi]\|$ for a strong solution.

Using the trace theorem, we naturally have that $||\gamma_0[\psi]|| \lesssim ||\psi||_{H^1(\mathbb R^2)}$. 
Then, to complete our estimates for a strong solution, we only need to  control: $$
\sup_{t \in [0,T]} ||\psi(t)||_{H^1(\mathbb R^2_+)} \le C(\mathcal E_0(0),\mathcal E_1(0),||\Delta_{\mu}\phi_0||).$$ This 
proceeds along the lines of the analysis characterizing $\mathscr D(\bA)$ (the proof of Theorem \ref{domainchar}). We outline the steps here, assuming $y_0 \in \mathscr D(\bA)$, so in particular we have that $[\Delta_\mu-U^2\partial_{x}^2]\phi \in L^2(\mathbb R^2_+)$ (recalling that $|U|<1$). At the energy level $\mathcal E_0$, we have control of $||\phi||_{H^1(\mathbb R^2_+)}$ and $||\psi||=||\phi_t+U\phi_x||$, from which control of $||\phi_t||$ follows. With the higher energy level $\mathcal E_1$ we obtain control of $||\phi_t||_{H^1(\mathbb R^2_+)}$ and $||\psi_t||=||\phi_{tt}+U\phi_{xt}||$. Combining these two yields control of $||\phi_{tt}||$. From the flow equation itself written out, ~$\phi_{tt}+2U\phi_{tx}+U^2\phi_{xx}= \Delta_\mu \phi$, we obtain control of $\Delta_U\phi\equiv \Delta_{\mu} \phi - U^2\partial_x^2 \phi \in L^2(\mathbb R^2_+)$. At this stage, elliptic theory with the KJC conditions proceeds exactly as in the proof of Theorem \ref{domainchar}, which yields the estimates:
\begin{equation}\label{H1}||\psi(t)||_{H^1(\mathbb R^2)}+||\nabla_{x,z} \phi_{x}(t)|| \lesssim C(\mathcal E_0(t),\mathcal E_1(t)).\end{equation}

Finally, we note that the initial quantities 
$$||w_t(0)||,~||w_{xx}(0)||,~||\nabla_{x,z}\phi(0)||,~||\phi_t(0)||,~||\psi(0)||=||\phi_t(0)+U\phi_x(0)||$$
are prescribed at the $\mathcal E_0$ level of solutions. 
To finalize our a priori estimate for strong solutions, we must use the assumptions on $y_0=(\phi(0),\phi_t(0);w(0),w_t(0)) \in \mathscr D(\mathbb A)$ to bound the quantities $\mathcal E(0)$ and $\mathcal E_1(0)$. Namely, we need to control:
$$||w_{tt}(0)||, ~||\psi_t(0)||=||\phi_{tt}(0)+U\phi_{xt}(0)||,~||\nabla_{x,z} \phi_t(0)||.$$
These follow from the equations themselves, since
\begin{align}
||w_{tt}(0) || \lesssim & ~||w_{xx}(0)||^2~||\partial_x^4 w(0)|| + ||w(0)||_{H^3(0,L)}^3+||\psi(0)||_{H^1(\mathbb R^2_+)} \\
||\psi_t(0)|| \lesssim & ~||\Delta_{\mu}\phi(0)||+U||\psi_x(0)||.
\end{align}
We can then invoke \eqref{H1} and
 note that, with $y_0 \in \mathscr D(\bA)$, we have $w(0) \in \mathscr D(\mathscr A)$, $w_t(0) \in \mathscr D(\mathscr A^{1/2})$, $\phi(0) \in H^1(\mathbb R^2_+)$ with $\Delta_{\mu}\phi(0) \in L^2(\mathbb R^2_+)$ and $\phi_t(0) \in H^1(\mathbb R^2_+)$---see Theorem \ref{domainchar}.
This concludes our a priori estimates for strong solutions, which are independent of $\delta \ge 0$---namely, the above estimates did not in any way utilize $\delta$ to control the quantities on the right hand side. These estimates are valid on some $[0,T^*]$, which is dictated by strictly by the beam parameters and not in any way on $\delta$. 
\end{proof}

\subsection{Limit process and completion of the proof of Theorem \ref{th:main2}}

From the family of fixed point solutions 
$\{(\phi^\delta,\psi^\delta; w^\delta,v^\delta) \in C([0,T^{\#}];\mathscr D(\bA_\delta))\}_{\delta>0},$
with estimates independent of $\delta$ in the space $Z_{T}$, one is able to construct a strong solution taking $\delta \searrow 0$, using the estimates in the previous section. We require that the initial data are in $\mathscr D(\bA)$. Most of the details of the construction are omitted here, since the corresponding construction for the beam is exactly carried out in \cite{maria2}, and the flow equation  is linear. 

We restrict to the interval $[0,T^{\#}]$, adjusting to a smaller subinterval if necessary. 
{\bf For each $\delta > 0$, the solution $(\phi^\delta, \psi^\delta; w^\delta, v^\delta)$ has sufficient regularity to satisfy the uniform a priori estimates in $Z_T$ established above (with $T$ independent of $\delta$)}. 
By weak compactness, we obtain subsequences with the following convergences:
\[
w^\delta \rightharpoonup w \quad \text{in } L^2(0,T; \mathscr D(\mathscr A)), \qquad 
w_t^\delta \rightharpoonup w_t \quad \text{in } L^\infty(0,T; H^2_*), \qquad
w_{tt}^\delta \rightharpoonup w_{tt} \quad \text{in } L^\infty(0,T; L^2).
\]
Following the same argument in \cite{maria2}, this yields a weak---and subsequently a strong---solution to:
\[
w_{tt} + D\partial_x^4 w = -F(w) + r_{(0,L)} \gamma_0[\psi].
\]
For the flow equation, weak limits being upgraded to strong is immediate, owing to linearity of that equation. 

As a final step, we remove the auxiliary parameter $\mu$, as is standard. 
This is accomplished using the fundamental theorem of calculus together with the estimate
\[
\|\phi(t)\|_{L^2(\mathbb R^2_+)} 
\le \|\phi_0\|_{L^2(\mathbb R^2_+)} 
+ \int_0^t \|\phi_t(\tau)\|_{L^2(\mathbb R^2_+)}\,d\tau,
\]
as in \cite{webster,supersonic}. 
Hence, if the initial data satisfy $\phi_0 \in H^1(\mathbb R^2_+)$, this property is recovered from the corresponding estimate on $\phi_t$, which is available in all cases above. 

Thus we have constructed a solution $y \in \mathcal{U}$, where $\mathcal{U}$ is given by \eqref{definitionofU} and is now entirely independent of $\delta$. Uniqueness follows from the established a priori energy estimates on the difference of two trajectories.

Finally, the $L^2(0,T;\mathscr D(\mathscr A))$ regularity of $w$ can be upgraded to $C([0,T];\mathscr D(\mathscr A))$ by returning to the beam equation itself: for each fixed $t \in [0,T]$, the already established regularity yields
\[
\|\partial_x^4 w\|^2 + ||w_x \partial_x^4 w||^2 
\le \epsilon \|\partial_x^4 w\|^2 
+ C_\epsilon \big( \|w_{tt}\|^2 + \|\psi\|_{H^1}^2 \big) 
+ C\big(\mathcal E_0(0), \mathcal E_1(0)\big).
\]
Here we again note the essential positive contribution of the nonlinear term $(w_x^2, (\partial_x^4 w)^2)$. This allows the absorption of lower-order terms into the left-hand side. Selecting $\epsilon > 0$ sufficiently small closes the estimate, thereby completing the proof of Theorem \ref{th:main2}.


\footnotesize
\begin{thebibliography}{99}

\bibitem{recentabhi} 
Balakrishna, A., Kukavica, I., Muha, B. and Tuffaha, A., 2024. 
Inviscid fluid interacting with a nonlinear two-dimensional plate. 
\emph{Interfaces and Free Boundaries}.

\bibitem{abhi} 
Balakrishna, A., Lasiecka, I. and Webster, J.T., 2023. 
Elastic stabilization of an intrinsically unstable hyperbolic flow-structure interaction on the 3D half-space. 
\emph{Mathematical Models and Methods in Applied Sciences (M3AS)}, 33(3), pp.505--545.

\bibitem{bal0} 
Balakrishnan, A.V., 2012. 
\emph{Aeroelasticity: Continuum Theory}. 
Springer Science \& Business Media.

\bibitem{balshub} 
Balakrishnan, A.V. and Shubov, M.A., 2008. 
Reduction of boundary value problem to Possio integral equation in theoretical aeroelasticity. 
\emph{Journal of Applied Mathematics}, pp.1--27.

\bibitem{ambal} 
Balakrishnan, A.V. and Tuffaha, A.M., 2012. 
Aeroelastic flutter in axial flow: The continuum theory. 
\emph{AIP Conference Proceedings}, 1493(1), pp.58--66.

\bibitem{bolotin}
Bolotin, V.V., 1963.  
\emph{Nonconservative Problems of Elastic Stability}. 
Pergamon Press, Oxford.

\bibitem{LBC96}
Boutet de Monvel, L. and Chueshov, I., 1996. 
Non-linear oscillations of a plate in a flow of gas. 
\emph{Comptes Rendus de l'Acad\'emie des Sciences. S\'erie I, Math\'ematique}, 322(10), pp.1001--1006.

\bibitem{b-c-1}
Boutet de Monvel, L. and Chueshov, I.D., 1999. 
Oscillations of von Karman's plate in a potential flow of gas. 
\emph{Izvestiya: Mathematics}, 63(2), pp.219--242.

\bibitem{chen-triggiani}
Chen, S. and Triggiani, R., 1990. 
Characterization of domains of fractional powers of certain operators arising in elastic systems, and applications. 
\emph{Journal of Differential Equations}, 88(2), pp.279--293.

\bibitem{survey1} 
Chueshov, I., Dowell, E.H., Lasiecka, I. and Webster, J.T., 2016. 
Nonlinear elastic plate in a flow of gas: Recent results and conjectures. 
\emph{Applied Mathematics \& Optimization}, 73(3), pp.475--500.

\bibitem{springer} 
Chueshov, I. and Lasiecka, I., 2010. 
\emph{Von Karman Evolution Equations: Well-posedness and Long-time Dynamics}. 
Springer, New York.

\bibitem{igorirena} 
Chueshov, I. and Lasiecka, I., 2012. 
Generation of a semigroup and hidden regularity in nonlinear subsonic flow-structure interactions with absorbing boundary conditions. 
\emph{Journal of Abstract Differential Equations and Applications}, 3, pp.1--27.

\bibitem{supersonic}
Chueshov, I., Lasiecka, I. and Webster, J.T., 2013. 
Evolution semigroups in supersonic flow-plate interactions. 
\emph{Journal of Differential Equations}, 254(4), pp.1741--1773.

\bibitem{K1}
Crighton, D.G., 1985. 
The Kutta condition in unsteady flow. 
\emph{Annual Review of Fluid Mechanics}, 17(1), pp.411--445.

\bibitem{maria1} 
Deliyianni, M., Gudibanda, V., Howell, J. and Webster, J.T., 2020. 
Large deflections of inextensible cantilevers: Modeling, theory, and simulation. 
\emph{Mathematical Modelling of Natural Phenomena}, 15, p.44.

\bibitem{maria2} 
Deliyianni, M. and Webster, J.T., 2021. 
Theory of solutions for an inextensible cantilever. 
\emph{Applied Mathematics \& Optimization}, 84(2), pp.1345--1399.

\bibitem{dowellnon}
Dowell, E.H., 1966--1967. 
Nonlinear oscillations of a fluttering plate, I and II. 
\emph{AIAA Journal}, 4(1966), pp.1267--1275; 5(1967), pp.1857--1862.

\bibitem{dowell1} 
Dowell, E.H., 2021. 
\emph{A Modern Course in Aeroelasticity}. 
Springer Nature.

\bibitem{DOWELL} 
Dunnmon, J.A., Stanton, S.C., Mann, B.P. and Dowell, E.H., 2011. 
Power extraction from aeroelastic limit cycle oscillations. 
\emph{Journal of Fluids and Structures}, 27(8), pp.1182--1198.

\bibitem{inext1} 
Dowell, E. and McHugh, K., 2016. 
Equations of motion for an inextensible beam undergoing large deflections. 
\emph{Journal of Applied Mechanics}, 83(5), p.051007.

\bibitem{energyharvesting} 
Erturk, A., 2011. 
\emph{Piezoelectric Energy Harvesting}, Vol. 2. 
Wiley \& Sons.

\bibitem{K2} 
Frederiks, W., Hilberink, H.C.J. and Sparenberg, J.A., 1986. 
On the Kutta condition for the flow along a semi-infinite elastic plate. 
\emph{Journal of Engineering Mathematics}, 20(1), pp.27--50.

\bibitem{Grisvard} 
Grisvard, P., 2011. 
\emph{Elliptic Problems in Nonsmooth Domains}. 
SIAM.

\bibitem{htw} 
Howell, J.S., Toundykov, D. and Webster, J.T., 2018. 
A cantilevered extensible beam in axial flow: Semigroup well-posedness and postflutter regimes. 
\emph{SIAM Journal on Mathematical Analysis}, 50(2), pp.2048--2085.

\bibitem{huang} 
Huang, L., 1995. 
Flutter of cantilevered plates in axial flow. 
\emph{Journal of Fluids and Structures}, 9(2), pp.127--147.

\bibitem{book} 
Kaltenbacher, B., Kukavica, I., Lasiecka, I., Triggiani, R., Tuffaha, A. and Webster, J.T., 2018. 
\emph{Mathematical Theory of Evolutionary Fluid-Flow Structure Interactions}. 
BirkhÃ€user.

\bibitem{piezosurvey} 
Kim, H.S., Kim, J.H. and Kim, J., 2011. 
A review of piezoelectric energy harvesting based on vibration. 
\emph{International Journal of Precision Engineering and Manufacturing}, 12, pp.1129--1141.

\bibitem{lagleug} 
Lagnese, J.E. and Leugering, G., 1991. 
Uniform stabilization of a nonlinear beam by nonlinear boundary feedback. 
\emph{Journal of Differential Equations}, 91(2), pp.355--388.

\bibitem{redbook}
Lasiecka, I. and Triggiani, R., 2000. 
\emph{Control Theory for Partial Differential Equations: Volume 1, Abstract Parabolic Systems: Continuous and Approximation Theories}. 
Cambridge University Press.

\bibitem{KJ}
Lasiecka, I. and Webster, J.T., 2014. 
Kutta-Joukowski flow conditions in flow-plate interactions: Subsonic case. 
\emph{Nonlinear Analysis: Real World Applications}, 7, pp.171--191.

\bibitem{LW} 
Lasiecka, I. and Webster, J.T., 2023. 
Flutter stabilization for an unstable, hyperbolic flow-plate interaction. 
In: \emph{Fluids Under Control}, pp.157--258. Springer International Publishing.


\bibitem{miyatake} 
Miyatake, S., 1973. 
Mixed problem for hyperbolic equation of second order. 
\emph{Journal of Mathematics of Kyoto University}, 13(3), pp.435--487.

\bibitem{miyatake2} 
Miyatake, S., 1993. 
Neumann operator for wave equation in a half space and microlocal orders of singularities along the boundary. 
\emph{Seminar on Partial Differential Equations} (Polytechnique, ``Goulaouic-Schwartz Seminar''), pp.1--6.

\bibitem{pazy} 
Pazy, A., 2012. 
\emph{Semigroups of Linear Operators and Applications to Partial Differential Equations}, Vol. 44. 
Springer Science \& Business Media.

\bibitem{sakamoto} 
Sakamoto, R., 1972. 
Mixed problems for hyperbolic equations. I. Energy inequalities. 
\emph{Matematika}, 16(1), pp.62--80.

\bibitem{savare} 
SavarÃ©, G., 1997. 
Regularity and perturbation results for mixed second order elliptic problems. 
\emph{Communications in Partial Differential Equations}, 22(5--6), pp.869--899.

\bibitem{inext2} 
Semler, C., Li, G.X. and PaÃ¯doussis, M.P., 1994. 
The non-linear equations of motion of pipes conveying fluid. 
\emph{Journal of Sound and Vibration}, 169(5), pp.577--599.

\bibitem{shubov} 
Shubov, M.A., 2006. 
Flutter phenomenon in aeroelasticity and its mathematical analysis. 
\emph{Journal of Aerospace Engineering}, 19(1), pp.1--12.

\bibitem{shubov*} 
Shubov, M.A., 2010. 
Solvability of reduced Possio integral equation in theoretical aeroelasticity. 
\emph{Advances in Differential Equations}, 15(9--10), pp.801--828.

\bibitem{tricomi} 
Tricomi, F.G., 1985. 
\emph{Integral Equations}, Vol. 5. 
Courier Corporation.

\bibitem{webster} 
Webster, J.T., 2011. 
Weak and strong solutions of a nonlinear subsonic flow-structure interaction: Semigroup approach. 
\emph{Nonlinear Analysis: Theory, Methods \& Applications}, 74(10), pp.3123--3136.

\end{thebibliography}
\end{document}